\documentclass[11pt,a4paper]{article}
\usepackage[pdftex]{graphicx}
\usepackage{amsmath,amssymb,amsfonts,amsthm}
\numberwithin{equation}{section}
\usepackage{indentfirst}
\usepackage{enumitem} 
\usepackage[margin=1.3in]{geometry}
\usepackage{fancyhdr}
\usepackage{makecell}

\usepackage[colorlinks=true,linkcolor=blue,citecolor=red,urlcolor=cyan]{hyperref}
\usepackage{titlesec}
\usepackage{etoolbox}
\usepackage{graphicx}  
\usepackage{changepage}
\theoremstyle{plain}
\newtheorem{theorem}{Theorem}[section]
\newtheorem{lemma}[theorem]{Lemma}
\newtheorem{corollary}[theorem]{Corollary}
\newtheorem{proposition}[theorem]{Proposition}
\newtheorem{definition}[theorem]{Definition}
\newtheorem{remark}{Remark}[section]

\newtheorem{conjecture}{Conjecture}[section]

\makeatletter
\renewcommand{\maketitle}{
	\begin{center}
		{\Large\bfseries{\@title}\par}
		\vskip 1em
		{\normalsize
			\lineskip .5em
			\begin{tabular}[t]{c}
				\@author
			\end{tabular}\par}
		\vskip 1.5em
	\end{center}
}
\makeatother
\renewenvironment{abstract}{
	\begin{adjustwidth}{1.3cm}{1.3cm}
		\noindent{\large\bfseries{A{\scriptsize BSTRACT.}}}
	}{
	\end{adjustwidth}
}

\usepackage{color}
\usepackage{xcolor}
\usepackage[normalem]{ulem} 
\usepackage{soul}

\usepackage{graphicx} 
\usepackage{xstring}  

\usepackage{xstring}
\usepackage{graphicx} 

\usepackage{marvosym}

\begin{document}
	
	\title{Overcolored Partition $k$-tuples Restricted by Parity of the Parts}

	\author{M. P. Thejitha and S. N. Fathima}
	
	\maketitle
	
	\begin{abstract}
	In this paper, we study the combinatorial object $\bar{b}^k_{r,s}(n)$ which counts the colored overpartition $k$-tuples wherein both even and odd parts occurs with $r$ and $s$ colors, respectively.  We establish several congruences which extends the results of Chacon and Sellers for infinite families of $r,s$ and $k$. In particular, we prove divisibility properties for $\bar{b}^k_{r,s}(n)$ modulo prime $p$ and derive new congruences modulo powers of $2$. The techniques involved to obtain our results rely on  theta function identities and modular forms.\\

		
	\end{abstract}
	
	\bigskip
	
	\vspace{0.5em}
	
	\section{Introduction}
	 \indent A partition of a positive integer $n$ is a finite sequence of non-increasing positive integers $\lambda_1, \lambda_2,..., \lambda_r$, called parts, such that $n=\sum_{i=1}^{r}\lambda_i$. We denote the number of such partitions by $p(n)$.
	For example, the 5 partitions of $n=4$ are
	\begin{align*}
		4,\; 3+1,\; 2+2,\; 2+1+1,\; 1+1+1+1.
	\end{align*}
	Thus, $p(4)=5$. The generating function for $p(n)$ is given by 
	\begin{eqnarray*}
		\sum_{n=0}^{\infty}p(n)q^n=\frac{1}{f_1},
	\end{eqnarray*}
	where here and for the sequel, for any complex numbers $a$ and $q$ with $|q|<1$, the standard $q$-product is defined by  $(a;q)_\infty:=\prod_{k=0}^{\infty}(1-aq^k)$. For notational convenience, we set $f_k^m:=(q^k;q^k)^m_\infty$, for $k,m\ge1$.
	
	\indent An overpartition of $n$, introduced by Corteel and Lovejoy \cite{corteel}, is a partition of $n$ where the first occurrence of a part may be overlined. For example, the overpartitions of $3$ are
	$ 3,\;\bar{3},\; 2+1,\; \bar{2}+1,2+\bar{1},\; \bar{2}+\bar{1},\; 1+1+1,\; \bar{1}+1+1$. Let $\bar{p}(n)$ denote the number of overpartitions of $n$. Thus, $\bar{p}(3)=8$. The generating function for $\bar{p}(n)$ is given by 
	\begin{align*}
		\sum_{n=0}^{\infty}\bar{p}(n)q^n=\dfrac{f_2}{f_1^2}.
	\end{align*}
	Further, Keister et al.\ \cite{keister} generalized the idea of overpartition to overpartition $k$-tuples. A partition $k$-tuple $(\pi_1, \pi_2,\dots, \pi_k)$ of weight $n$ is a $k$-tuple of partitions $\pi_1, \pi_2,\dots, \pi_k$ such that their sum equals $n$. A partition $k$-tuple of $n$ is overpartition $k$-tuple if each $\pi_i$ is overpartition. The number of overpartition $k$-tuples of $n$ is denoted by $\bar{p}_k(n)$, with $\bar{p}_k(0):=1$ and the generating function for $\bar{p}_k(n)$ is 
	\begin{align*}
		\sum_{n=0}^{\infty}\bar{p}_k(n)q^n=\dfrac{f_2^k}{f_1^{2k}}.
	\end{align*}
	\indent The extensive study of $\bar{p}(n)$ for the last two decades has led to the investigation of broader classes of overpartitions. One such overpartition studied by Kim \cite{kim}, is the overcubic partition of a positive integer $n$, denoted by $\bar{b}_{2,1}(n)$,  that counts the partitions where even parts may occur in two colors and the first occurrence of parts may be overlined. The generating function for $\bar{b}_{2,1}(n)$  is given by
	\begin{align*}
		\sum_{n=0}^{\infty}\bar{b}_{2,1}(n)q^n=\dfrac{f_4}{f_1^2f_2}.
	\end{align*}
	Thus, $\bar{b}_{2,1}(2)=6$ with the 6 overcubic partitions of $2$ given by $\bar{2}_1,\;2_1,\;\bar{2}_2, \;2_2, \bar{1}+1,\;1+1 $.
	Prior to Kim's work, Chan \cite{chan} introduced the notion of cubic partition of a positive integer $n$ where even parts may appear in two colors. 	A partition $k$-tuple $(\pi_1, \pi_2,\dots, \pi_k)$ is overcubic partition $k$-tuple if each $\pi_i$ is overcubic partition. The arithmetic properties of overcubic partition $k$-tuples was studied extensively by many researchers for several values of $k$ (for more details we refer the reader to see \cite{nsaikia,chen,kim2,mnaika,ssnay,cray,mpsai}). We denote overcubic partition $k$-tuples by $\bar{b}^k_{2,1}(n)$ and it is easy to note the generating function for $\bar{b}^k_{2,1}(n)$ is
	\begin{align}
		\sum_{n=0}^{\infty}\bar{b}^k_{2,1}(n)q^n=\dfrac{f_4^k}{f_1^{2k}f_2^k}.
	\end{align}
	Very recently, Chacon and Sellers \cite{chacon} proved new infinite families of congruences for overcubic partition $k$-tuples. For example, they proved the following theorem which characterizes the number of overcubic partition $k$-tuples modulo 4.
	\begin{theorem}[{{\cite[Theorem~3.3]{chacon}}}]
		We have 
		\begin{enumerate}\label{6tsell}
			\item For all $n\ge1$ and $k$ even, $b_{2,1}^k(n)\equiv 0\pmod{4}$.
			\item For all $n\ge1$ and $k$ odd,
			\[
			\bar{b}^k_{2,1}(n)\equiv
			\begin{cases}
				2\pmod 4, & \text{if } n=\ell^2 \text{ or } n=2\ell^2\\[4pt]
				0 \pmod 4, & \text{otherwise }.
			\end{cases}
			\]
		\end{enumerate}
	\end{theorem} 
	Chacon and Sellers \cite{chacon} also proved several congruences for $\bar{b}^k_{2,1}(n)$ involving odd moduli. For instance, they proved the following theorems.
	\begin{theorem}[{\cite[Th. 4.1 and 4.2]{chacon}}]\label{7tc1}
		For all $n\ge0$, 
		\begin{align*}
			\bar{b}_{2,1}^4(14n+7)&\equiv 0\pmod{7}\\
			\bar{b}_{2,1}^4(22n+11)&\equiv 0\pmod{11}.
		\end{align*}
	\end{theorem}
	\begin{theorem}[{{\cite[Th. 4.3]{chacon}}}]\label{7tc2}
		For all $n\ge1$, $\bar{b}_{2,1}^8(5n)\equiv 0\pmod{5}$.
	\end{theorem}
	\noindent Subsequently, Maity and Saikia \cite{maity} obtained infinite families of congruences for $\bar{b}^k_{2,1}(n)$ modulo powers of $2$.\\
	\indent With the goal to generalize cubic partitions, Amdeberhan et.\ al \cite{amd} introduced generalized overcubic partitions where each even part may appear in $r\ge1$ colors and the first occurrence of parts may be overlined. Let $\bar{b}_{r,1}(n)$ denote the number of such partitions. The generating function for $\bar{b}_{r,1}(n)$ is given by 
	\begin{align*}
		\sum_{n=0}^{\infty}\bar{b}_{r,1}(n)q^n=\dfrac{f_4^{r-1}}{f_1^2f_2^{2r-3}}.
	\end{align*}
	As a natural generalization of $\bar{b}_{r,1}(n)$, the authors \cite{thej4} introduced a partition wherein both even and odd parts may appear in one of $r$-colors and $s$-colors, respectively, for fixed $r,s\ge1$ and the first occurrence of a part may be overlined. For example, $\bar{b}_{1,2}(2)=10$, since the partitions in question are: $2, \bar{2}, 1_1+1_1, \bar{1}_1+1_1, \bar{1}_1+1_2, 1_1+\bar{1}_2, \bar{1}_1+\bar{1}_2, 1_2+1_2, \bar{1}_2+1_2$. Let $\bar{b}_{r,s}(n)$ denote the number of such partitions. The generating function for $\bar{b}_{r,s}(n)$ is given by 
	\begin{align*}
	\bar{B}_{r,s}(q):=	\sum_{n=0}^{\infty}\bar{b}_{r,s}(n)q^n=\dfrac{f_2^{3s-2r}}{f_1^{2s}f_4^{s-r}}.
	\end{align*}
	
	Motivated by the works above, we now consider overcolored partition $k$-tuples restricted by parity of parts, which is a $k$-tuple analogue of $\bar{b}_{r,s}(n)$. Let $\bar{b}^k_{r,s}(n)$ denote the number of such partitions. The generating function for $\bar{b}^k_{r,s}(n)$ is given by
	\begin{align}\label{6gf}
		\bar{B}^k_{r,s}(q):=\sum_{n=0}^{\infty}\bar{b}^k_{r,s}(n)q^n=\dfrac{f_2^{(3s-2r)k}}{f_1^{2sk}f_4^{(s-r)k}}.
	\end{align}
	The partition function $\bar{b}_{r,s}^k(n)$ can also be viewed as colored overpartitions with $rk$ and $sk$ colored even and odd parts, respectively, studied by the authors in \cite{thej4}. Also, we note that for $r=s$ in \eqref{6gf}, $\bar{b}_{r,s}^k(n)=\bar{p}_{rk}(n)$.
	
	\indent	The purpose of this paper is to extend the study of arithmetic properties of colored overpartitions to generalized colored overpartition $\bar{b}^k_{r,s}(n)$. We prove divisibility properties and new infinite families of congruences for $\bar{b}^k_{r,s}(n)$ modulo prime and powers of 2, which extends the work of Chacon and Sellers \cite{chacon} and provide generalizations of several results proved by them.  \\
	First, we recall Ramanujan's theta functions $\phi(q)$ and $\psi(q)$ which are defined as
	\begin{align*}
		\phi(q)&:=1+2\sum_{n=1}^{\infty}q^{n^2}\\
		\psi(q)&:=\sum_{n=0}^{\infty}q^{(n^2+n)/2}.
	\end{align*}
	\noindent In \cite{thej4}, the authors also proved the following functional equation for $\bar{B}_{r,s}(q)$ which is a generalization of \cite[Corollary~2.3]{seller1} and \cite[Lemma~4.1]{das}.
	\begin{theorem}[{{\cite[Lemma~3.1]{thej4}}}]\label{6t1}
		For all $r,s\ge1$, we have
		\begin{align*}
			\bar{B}_{r,s}(q)=\phi(q)^s\prod_{i=1}^{\infty}\phi(q^{2^i})^{(r+s)\cdot2^{i-1}},
		\end{align*}
		where $\phi(q)=\frac{f_2^5}{f_1^2f_4^2}$.
	\end{theorem}
	We present the following congruences. It is easy to observe, Theorem \ref{6tmod4} is a generalization of Theorem \ref{6tsell}.
	\begin{theorem}\label{6tmod4}
		We have
		\begin{enumerate}
			\item For all $n,r,s\ge1$ and $k$ even, $\bar{b}^k_{r,s}(n)\equiv 0\pmod{4}$.
			\item For all $n,r,s\ge 1$, some integer $\ell\ge1$, and $k$ odd,
			\[
			\bar{b}^k_{r,s}(n)\equiv
			\begin{cases}
				2s\pmod 4, & \text{if } n=\ell^2 \\[4pt]
				2(r+s)\pmod 4, & \text{if } n=2\ell^2\\[4pt]
				0 \pmod 4, & \text{otherwise }.
			\end{cases}
			\]
		\end{enumerate}
	\end{theorem}
	
	In the following theorem, we characterize $\bar{b}^k_{r,s}(n)$ modulo $8$ which generalizes {{\cite[Theorem 1.4]{thej4}}} for $k$-tuples.
	\begin{theorem}\label{6tmod8}
		We have
		\begin{enumerate}
			\item For all $n,r,s\ge 1$, some integer $\ell\ge 1$ and $k$ even,
			\[
			\bar{b}^k_{r,s}(n)\equiv
			\begin{cases}
				2rk\pmod 8, & \text{if } n=(2\ell)^2 \text{ or } 2(2\ell-1)^2,\\[4pt]
				2sk\pmod 8, & \text{if } n=(2\ell-1)^2,\\[4pt]
				6rk\pmod 8, & \text{if } n=2(2\ell)^2,\\[4pt]
				0\pmod 8,  & \text{otherwise }.\\[4pt]
			\end{cases}
			\]
			\item For all  $n,r,s\ge 1$, some integers $\ell, m\ge 1$ and $k$ odd,
			\[
			\bar{b}^k_{r,s}(n)\equiv
			\begin{cases}
				2sk\pmod 8, & \text{if } n=(2\ell-1)^2,\\[4pt]
				2k\left(s+(r+s)\left((r+s)k+1\right)\right)\pmod 8, & \text{if } n=(2\ell)^2,\\[4pt]
				2k(r+s^2k)\pmod 8, & \text{if } n=2(2\ell-1)^2,\\[4pt]
				2k(3r+s^2k+2s)\pmod 8,  & \text{if } n=2(2\ell)^2,\\[4pt]
				4s(r+s)\pmod 8,  & \text{if } n=\ell^2+2m^2,\\[4pt]
				0\pmod 8,  & \text{otherwise }.\\[4pt]
			\end{cases}
			\]
		\end{enumerate}
	\end{theorem}
	In addition to Ramanujan-type congruences, the distribution of the coefficients of a formal power series modulo $M$ of partition functions has been studied extensively. Let $A(q):=\sum_{n=0}^{\infty}a(n)q^n$ be an integral power series and let $0\le r\le M$. The arithmetic density $\delta_r(A,M;X)$ is defined by
	\begin{align*}
		\delta_r(A,M;X)=\frac{\#\{n\le X: a(n)\equiv r\pmod{M}\}}{X}.
	\end{align*}
	An integral power series $A$ is called lacunary modulo $M$ if 
	\begin{align*}
		\lim_{X\to\infty}\delta_0(A,M;X)=1,
	\end{align*}
	equivalently, if almost all the coefficients of $A$ are divisible by $M$. In \cite{das}, Das et. al have proved divisibility results for generalized overcubic partitions. In the theorem below, we extend the divisbility result for $\bar{b}^k_{r,s}(n)$.
	\begin{theorem}\label{6tlac1}
		Let $n\ge 0$ and let $j,r,s,k,\alpha,m\ge 1$ be integers with $gcd(2,m)=1$. If $r=2^\alpha m$ or $k=2^\alpha m$, then the set
		\begin{align*}
			\{n\in\mathbb{N}:\bar{b}^k_{r,s}(n)\equiv 0\pmod{2^j}\} 
		\end{align*}
		has arithmetic density $1$, namely, 
		\begin{align*}
			\lim_{X\to\infty}\frac{\#\{0\le n\le X:\bar{b}^k_{r,s}(n)\equiv 0\pmod{2^j}\}}{X}=1.
		\end{align*}
	\end{theorem}
	\begin{theorem}\label{6tdiv1}
		Let $j\ge2$, $r,s,k,\alpha,m\ge 1$ be integers with $gcd(2,m)=1$ and prime $p\ge 3$ such that $p^{j-1}(p^2-1)\ge 3k(r+s)$. If $r=2^\alpha m$ or $k=2^\alpha m$, then the set
		\begin{align*}
			\{n\in\mathbb{N}:\bar{b}^k_{r,s}(n)\equiv 0\pmod{p^j}\} 
		\end{align*}
		has arithmetic density $1$, namely, 
		\begin{align*}
			\lim_{X\to\infty}\frac{\#\{0\le n\le X:\bar{b}^k_{r,s}(n)\equiv 0\pmod{p^j}\}}{X}=1.
		\end{align*}
	\end{theorem}
	We employ a result due to Ono and Taguchi \cite{taguchi} to prove Theorem \ref{6tdiv2}.
	\begin{theorem}\label{6tdiv2}
		Let $n$ be a non-negative integer and let $j,r,s,k,\alpha,m\ge 1$ be integers with $gcd(2,m)=1$. If $r=2^\alpha m$ or $k=2^\alpha m$, and $2^j>rk$, then there exists an integer $t\ge0$ such that for every $u\ge 1$ and distinct primes $p_1,p_2, \dots, p_{t+u}$ greater than $3$, we have
		\begin{align*}
			\bar{b}^k_{r,s}\left(\dfrac{p_1\cdots p_{t+u}\cdot n}{24}\right)\equiv 0\pmod{2^j},
		\end{align*}
		for every $n\ge 1$ coprime to $p_1,p_2, \dots, p_{t+u}$.
	\end{theorem}
	In the following theorem, we prove infinite family of congruences modulo powers of $2$ using the theory of Hecke eigenforms.
	\begin{theorem}\label{6thec1}
		Let $j$ and $n$ be non-negative integers and let $k,\ell\ge0$, $r, \alpha,\beta\ge 1$ be integers such that $\alpha$ and $\beta$ are odd. Let $p_1, p_2,\dots,p_{j+1}$ be primes such that $p_i\ge 3$ and $p_i\not\equiv1\pmod 4$ for each $1\le i\le j+1$. For any integer $t\not\equiv 0\pmod{p_{j+1}}$, we have
		\begin{align*}
			\bar{b}^{2^\ell\beta}_{r, 2^k\alpha}\left(2p_1^2p_2^2\cdots p_j^2p_{j+1}^2+p_1^2p_2^2\cdots p_j^2p_{j+1}(t+p_{j+1})\right)\equiv 0\pmod{2^{k+\ell+2}}.
		\end{align*}
	\end{theorem}
	The following corollary can be easily obtained by setting $p_1=p_2=\dots=p_{j+1}=p$ in Theorem \ref{6thec1}.
	\begin{corollary}
		Let $j$ and $n$ be non-negative integers and let $k,\ell\ge0$, $r, \alpha,\beta\ge 1$ be integers such that $\alpha$ and $\beta$ are odd. Let $p_i\ge 3$  be a prime with $p\equiv 3\pmod{4}$. For any integer $t\not\equiv0\pmod p$, we have
		\begin{align*}
			\bar{b}^{2^\ell\beta}_{r, 2^k\alpha}\left(2p^{2j+2}n+p^{2j+1}t+p^{2j+2}\right)\equiv 0\pmod{2^{\ell+k+2}}.
		\end{align*}
	\end{corollary}
		Furthermore, we extend Theorems \ref{7tc1} and \ref{7tc2} to $\bar{b}_{r,s}^k(n)$ for certain values of $r$, $s$ and $k$ using Sturm's theorem.
	\begin{theorem}\label{7t1.1}
		For all $n\ge 0$ and  $(p,r,s,k)\in\{(5,5,3,4), (5,2,4,2), (5,3,4,2),\\
		(5,5,4,3), (5,1,3,8),(5,1,8,3), (5,1,8,1), (7,2,3,4), (7,5,4,2), (7,5,4,3), (7,3,8,1),\\ (7,5,8,2), (11,2,2,4),
		(11,2,3,4), (11,4,3,4), (11,1,4,5), (11,1,4,10), (11,4,1,4)\}$,\\ we have
		\begin{align*}
			\bar{b}_{r,s}^k(2pn+p)\equiv 0\pmod{p}.
		\end{align*}
	\end{theorem}
	\begin{theorem}\label{7t1.2}
		For all $n\ge 1$ and $(p,r,s,k)\in \{(5,4,4,4), (5,3,1,8), (7,3,1,8),\\ (7,3,2,8), (7,3,3,8), (11,5,1,8),(11,5,2,8),(11,5,3,8), (11,5,4,8), (11,5,5,8),\\ (13,6,1,8), (13,6,2,8), (13,6,3,8), (13,6,4,8), (13,6,5,8), (13,6,6,8), (17,8,1,8),\\ (17,8,2,8), (17,8,3,8), (17,8,4,8), (17,8,5,8), (17,8,6,8), (17,8,7,8), (17,8,8,8),\\ (19,9,1,8), (19,9,2,8), (19,9,3,8), (19,9,4,8), (19,9,5,8), (19,9,6,8), (19,9,7,8),\\ (19,9,8,8), (19,9,9,9)
		\}$, we have
		\begin{align*}
			\bar{b}_{r,s}^k(pn)\equiv 0\pmod{p}.
		\end{align*}
	\end{theorem}
	\begin{theorem}\label{7t1.3}
		For all $n\ge 0$ and $(p,r,s,k)\in\{(5,5,3,4), (5,2,4,2), (5,1,2,4),\\ (7,2,3,4),(7,4,6,2)\}$, we have
		\begin{align*}
			\bar{b}_{r,s}^k(pn)\equiv \bar{b}_{r,s}^k (p^2n)\pmod{p}.
		\end{align*}
	\end{theorem}
	The following result establishes Ramanujan-type congruence modulo any prime $p\ge3$ for $\bar{b}^k_{r,s}(n)$ for infinitely many $r$ and $s$.
	\begin{theorem}\label{6tmodp1}
		For all $n\ge0$, $m\ge j\ge 0$, $k,s\ge 1$ and prime $p\ge 3$, with $1\le r\le p-1$, such that $r$ is a quadratic non-residue modulo $p$ and $sk\equiv 1\pmod{p}$,
		\begin{align*}
			\bar{b}^{k}_{p(m-j)+p-s,pm+s}(pn+r)\equiv 0\pmod{p}.
		\end{align*}
	\end{theorem}
	\begin{remark}
		The congruence $\bar{b}^{k}_{2,1}(3n+2)\equiv 0\pmod{3}$ proved in \cite[Theorem 4.7, 4.10 and 4.11]{chacon} for $k=27l+10,\; 9l+7$ and $3l+1$ can be directly obtained from the above result.
	\end{remark}
	We observe the following theorem is a generalization of a result of Chacon and Sellers {{\cite[Theorem 5.1]{chacon}}} modulo any prime $p$.
	\begin{theorem}\label{6tmodp2}
		For prime $p$, integers  $n\ge0$, $r,s, k\ge1$ and $j$ such that $1\le j\leq p-1$, we have 
		\begin{align*}
			\bar{b}^{p}_{r,s}(pn+j)\equiv 0\pmod{p}.
		\end{align*}
	\end{theorem}
	We also prove the following congruences modulo $3$ for infinitely many $r$, $s$ and $k$.
	\begin{theorem}\label{6tmod3.1}
		For all $n,r,s,k\ge 0$, we have
		\begin{align}
			\bar{b}^k_{3r,3s}(3n+1)&\equiv \bar{b}^k_{3r,3s}(3n+2)\equiv 0\pmod{3}\label{6e3n1.1}\\
			\bar{b}^{3k+2}_{3r+1,3s}(3n+1)&\equiv\bar{b}^{3k+1}_{3r+2,3s}(3n+1)\equiv 0\pmod{3}\label{6e3n1.2}\\
			\bar{b}^{3k+2}_{3r+1,3s+1}(3n+2)&\equiv\bar{b}^{3k+1}_{3r+2,3s+2}(3n+2)\equiv 0\pmod{3}\label{6e3n2.1}.
		\end{align}
	\end{theorem}
	\begin{theorem}\label{6tmod3.2}
		For all $n,r,s,k\ge 0$, $0\le t\le 2$, $0\le u\le 2$, we have
		\begin{align}
			\bar{b}^{9k+3}_{9r+3t+2,9s+3u}(9n+3)&\equiv\bar{b}^{9k+6}_{9r+3t+1,9s+3u}(9n+3)\equiv 0\pmod{3}\label{6e9.1}\\
			\bar{b}^{9k+3}_{9r+3t+2,9s+3u+2}(9n+6)&\equiv\bar{b}^{9k+6}_{9r+3t+1,9s+3u+1}(9n+6)\equiv 0\pmod{3}\label{6e9.2}\\
			\bar{b}^{9k+3}_{9r+3t+2,9s+3u+1}(9n+6)&\equiv\bar{b}^{9k+6}_{9r+3t+1,9s+3u+2}(9n+6)\equiv 0\pmod{3}.\label{6e9.3}
		\end{align}
	\end{theorem}
	In order to state congruences $\bar{b}_{r,s}^k(9n+3)\equiv 0\pmod{3}$ for certain $r,s,$ and $k$, we define the following sets where each triple represents values of $(t,u,\ell)$:
	\begin{align*}
		A_{t,u,\ell}:=&\{(1,2,5),(7,5,2),(4,8,8),(5,1,1),(2,4,7),(8,7,4)\}\\
		B_{t,u,\ell}:=&\{(4,1,8),(7,4,2),(1,7,5),(8,2,4),(2,5,7),(5,8,1)\}\\
		C_{t,u,\ell}:=&\{(2,0,7),(5,0,1),(8,0,4),(1,0,5),(4,0,8),(7,0,2)\}\\
		D_{t,u,\ell}:=&\{(2,1,2),(8,4,5), (5,7,8), (4,2,1), (1,5,4), (7,8,7)\}\\
		E_{t,u,\ell}:=&\{(7,1,7), (1,4,4), (4,7,1), (5,2,8), (8,5,5), (2,8,2)\}.
	\end{align*}
	\begin{theorem}\label{6tmod3.3}
		For all $n,r,s,k,t\ge0$ and values of $(t,u,\ell)$ as in sets $A_{t,u,\ell}, B_{t,u,\ell}, C_{t,u,\ell},\\ D_{t,u,\ell}$ and $E_{t,u,\ell}$, we have
		\begin{align*}
			\bar{b}^{9k+\ell}_{9r+t,9s+u}(9n+3)\equiv 0\pmod{3}.
		\end{align*}
	\end{theorem}
	\begin{remark}
		The congruence proved in \cite[Theorem~4.9]{chacon} can be obtained from the above result for $r=0$, $s=0$ and $(t,u,\ell)=(2,1,2)$.
	\end{remark}
	Similarly, to state congruences $\bar{b}_{r,s}^k(9n+6)\equiv 0\pmod{3}$ for certain $r,s,$ and $k$, we define the following sets where each triple represents values of $(t,u,\ell)$:
	\begin{align*}
		G_{t,u,\ell}:=&\{(4,2,8), (1,5,5), (7,8,2), (2,1,7), (8,4,4), (5,7,1)\}\\
		H_{t,u,\ell}:=&\{(7,2,2), (4,5,8), (1,8,5), (8,1,4), (5,4,1), (2,7,7)\}\\
		I_{t,u,\ell}:=&\{(1,1,5), (4,4,8), (7,7,2), (2,2,7), (5,5,1), (8,8,4)\}\\
		J_{t,u,\ell}:=&\{(7,1,2), (1,4,5), (4,7,8), (5,2,1), (8,5,4), (2,8,7)\}\\
		K_{t,u,\ell}:=&\{(2,6,7), (5,6,1), (8,6,4), (1,3,5), (4,3,8), (7,3,2)\}\\
		L_{t,u,\ell}:=&\{(2,3,7), (5,3,1), (8,3,4), (1,6,5), (4,6,8), (7,6,2)\}\\
		M_{t,u,\ell}:=&\{(8,1,5), (5,4,8), (2,7,2), (4,2,7), (4,5,1), (1,8,4)\}\\
		N_{t,u,\ell}:=&\{(5,1,8), (2,4,2), (8,7,5), (1,2,4), (7,5,7), (4,8,1)\}\\
		O_{t,u,\ell}:=&\{(4,1,1), (7,4,7), (1,7,4), (8,2,5), (2,5,2), (5,8,8)\}\\
		P_{t,u,\ell}:=&\{(1,1,4), (4,4,1), (7,7,7), (2,2,2), (5,5,8), (8,8,5)\}.
	\end{align*}
	\begin{theorem}\label{6tmod3.4}
		For all $n,r,s,k,t\ge0$ and values of $(t,u,\ell)$ as in sets $G_{t,u,\ell}, H_{t,u,\ell}, I_{t,u,\ell},\\ J_{t,u,\ell}, K_{t,u,\ell}, L_{t,u,\ell}, M_{t,u,\ell}, N_{t,u,\ell}, O_{t,u,\ell},$ and $P_{t,u,\ell}$, we have
		\begin{align*}
			\bar{b}^{9k+\ell}_{9r+t,9s+u}(9n+6)\equiv 0\pmod{3}.
		\end{align*}
	\end{theorem}
	The rest of the paper is organized as follows. Section \ref{6s2} provides a brief summary of key definitions and results required for subsequent sections. The proofs of our main results are presented in Sections \ref{6s3}--\ref{6s5}, and Section \ref{6s7} concludes the paper with interesting remarks.
	
	\section{Preliminaries}\label{6s2}
	We begin this section by recalling the prerequisites concerning $q$-series and modular forms that will be used in the proofs of our main theorems. For an introduction to $q$-series and modular forms, two standard references are \cite{gasper} and \cite{koblitz}, respectively. We first recall $q$-series identities and dissection formulas.
	\begin{theorem}[{{\cite[Equation~(1.6.1)]{poq}}}]
		We have
		\begin{align*}
			f_1=(q;q)_\infty=\sum_{k=-\infty}^{\infty}(-1)^kq^{\frac{3k^2+k}{2}}.
		\end{align*}
	\end{theorem}
	\begin{lemma}[{{\cite[Lemma~2.2]{chacon}}}]\label{6l1}
		We have
		\begin{align*}
			(-q;-q)_\infty=\dfrac{f_2^3}{f_1f_4}.
		\end{align*}
	\end{lemma}
	The following lemma is a consequence of binomial theorem and will be used repeatedly in the sequel. 
	\begin{lemma}\label{bt}
		For any prime $p$ and integers $\ell\ge0$ and $a\ge1$, we have
		\begin{align*}
			f_\ell^{p^a} \equiv f_{\ell p}^{p^{a-1}} \pmod{p^a}.
		\end{align*}
	\end{lemma}
	\begin{lemma}\label{ld}
		The following $2$-dissection holds:
		\begin{align}
			\dfrac{1}{f_1^2}&=\dfrac{f_8^5}{f_2^5f_{16}^2}+2q\dfrac{f_4^2f_{16}^2}{f_2^5f_8} \label{edf2}.	
		\end{align}
	\end{lemma}
	\begin{proof}
		Equation $\eqref{edf2}$ is same as {\cite[Eq. 1.9.4]{poq}}.
	\end{proof}
	\begin{lemma}\label{6l3d}
		We have
		\begin{align}
			\psi(q)&=\frac{f_2^2}{f_1}\label{epsi}\\
			\psi(-q)&=\dfrac{f_1f_4}{f_2}\label{epsim}\\
			P(q)&=\dfrac{f_2f_3^2}{f_1f_6}\label{ep}\\
			P(-q)&=\dfrac{f_1f_4f_6^5}{f_2^2f_3^2f_{12}^2}\label{epm}\\
			\psi(q)&=P(q^3)+q\psi(q^9)\label{epsip}\\
			\dfrac{1}{\psi(q)}&=\dfrac{\psi(q^9)}{\psi(q^3)^4}\left(P(q^3)^2-qP(q^3)\psi(q^9)+q^2\psi(q^9)^2\right)\label{epsir}\\
			F(q)&=F(q^9)\left(X(q^3)^{-1}-q-2q^2X(q^3)\right), \label{ef}
		\end{align}
		where $P(q)=\displaystyle\sum_{n=-\infty}^{\infty}q^{\frac{3n^2+n}{2}}$,  $F(q)=f_1f_2$ and $X(q)=\dfrac{f_1f_6^3}{f_2f_3^3}$.
	\end{lemma}
	\begin{proof}
		Equation \eqref{epsi} is same as the equation of Hirschhorn \cite[ (1.5.7)]{poq}. For equations \eqref{ep}, \eqref{epsip} and \eqref{ef} see \cite[(14.3.3),(26.1.2) and (14.4.5)]{poq}. Equations \eqref{epsim} and \eqref{epm} are obtained by Lemma \ref{6l1}. For equation \eqref{epsir} see Hirschhorn and Sellers \cite[Lemma 2.2]{hs1}.
	\end{proof}

	We now end this section with necessary definitions and facts on arithmetic properties of integer weight modular forms to keep the paper self-contained. Let $\mathbb{H}$ denote the complex upper-half plane. For a positive integer $k$, the complex vector space of modular forms of weight $k$ with respect to congruence subgroup $\Gamma$ will be denoted by $M_k(\Gamma)$ (for more details see \cite{wom}).
	\begin{definition}[{{\cite[Definition~1.15]{wom}}}]
		Let $\chi$ be a Dirichlet character modulo $N$. Then a modular form $f\in M_k(\Gamma_1(N))$ has Nebentypus character $\chi$ if
		\begin{align*}
			f\left(\dfrac{az+b}{cz+d}\right)=\chi(d)(cz+d)^kf(z),
		\end{align*} 
		for all $z\in \mathbb{H}$ and all $\begin{bmatrix}
			a & b \\
			c & d
		\end{bmatrix}\in\Gamma_0(N)$. We denote the space of such modular forms by $M_k(\Gamma_0(N),\chi)$.
	\end{definition}
	
	\indent The Dedekind's eta-function $\eta(z)$ defined by
	\begin{align}\label{2.1}
		\eta(z):=q^{1/24}(q;q)_\infty=q^{1/24}\prod_{n=1}^\infty(1-q^{n}), \text{ where } q=e^{2\pi iz},
	\end{align}
	which is a non-vanishing holomorphic function on $\mathbb{H}$. A function is called an eta-quotient if it is of the form
	\begin{align}
		f(z)=\prod_{\delta\mid N}\eta(\delta z)^{r_\delta},
	\end{align}
	where $N$ is a positive integer and $r_{\delta}$ is an integer.\\ 
	\indent If $f(z)$ is an eta-quotient associated with positive integer weight $k$ satisfying the conditions of the following theorem and is holomorphic at all the cusps of $\Gamma_0(N)$, then $f(z) \in M_k(\Gamma_0(N), \chi )$. 
	\begin{theorem}[{{\cite[Theorem~1.64]{wom}}}]\label{t2.1}
		If $ f(z)= \prod_{\delta \mid N} \eta(\delta z)^{r_\delta}$ is an eta-quotient with
		$k= \frac{1}{2} \sum_{\delta \mid N}{r_\delta} \in \mathbb{Z}$, and satisfies the following additional properties:
		\begin{align}\label{2.3}
			\sum_{\delta\mid N} \delta {r_\delta} \equiv 0 \pmod {24},
		\end{align}
		\begin{align}\label{2.4}
			\sum_{\delta \mid N} \frac{N}{\delta}  {r_\delta} \equiv 0 \pmod {24},
		\end{align}
		then $f(z)$ satisfies 
		\begin{align*}
			f\left(\dfrac{az+b}{cz+d}\right)=\chi(d)(cz+d)^kf(z)
		\end{align*} 
		for every $\begin{bmatrix}
			a & b \\
			c & d
		\end{bmatrix}\in \Gamma_0(N)$, where the character $\chi$
		is defined by
		$\chi (d) := \bigg( \frac{(-1)^k \prod_{\delta \mid N} \delta^{r_\delta}}{d} \bigg).$
	\end{theorem}
	To check the holomorphicity of $f(z)$ at cusps of $\Gamma_0(N)$, it is enough to check that the orders at the cusps are non-negative. The next theorem helps us to determine orders of an eta-quotient at the cusps of $\Gamma_0(N)$.
	\begin{theorem}[{{\cite[Theorem~1.65]{wom}}}]\label{t2.2}
		Let $c,d,$ and $N$ be positive integers with $d\mid N$ and $gcd(c,d)=1$. If $f(z)$ is an eta-quotient satisfying the conditions of Theorem \ref{t2.1} for $N$, then the order of vanishing of $f(z)$ at the cusp $\frac{c}{d}$ is 
		\begin{align*}
			\dfrac{N}{24}\sum_{\delta\mid N} \frac{gcd(d,\delta)^2 r_\delta}{ gcd(d,\frac{N}{d})d\delta}.		
		\end{align*}
	\end{theorem}
	We recall the following definitions of the Hecke operator and Hecke eigenform which are crucial for our proofs.
	\begin{definition}[{{\cite[Definition~2.1]{wom}}}]
		Let $m$ be a positive integer and $f(z) = \sum_{n=0}^ \infty a(n)q^n \in  M_k(\Gamma_0(N), \chi ).$ The Hecke operator $T_m$ acts on $f(z)$ by 
		\begin{align}
			f(z)\mid {T_m} := \sum_{n=0}^\infty \bigg( \sum_{d\mid gcd(n,m)} \chi (d) d^{k-1}a \bigg(\frac{nm}{d^2} \bigg) \bigg)q^n.
		\end{align}
		In particular, if $m=p$ is a prime, then 
		\begin{align}\label{2.7}
			f(z)\mid {T_p} := \sum_{n=0}^\infty \bigg( a(pn)+ \chi (p) p^{k-1}a \bigg(\frac{n}{p} \bigg) \bigg)q^n.
		\end{align}
		We adopt the convention that $a(n/p)=0$ whenever $p\nmid n$.
	\end{definition}
	\begin{definition}[{{\cite[Definition~2.5]{wom}}}] 
		A modular form $f(z)\in M_k(\Gamma_0(N), \chi)$ is called a Hecke eigenform if for every $m\ge 2$ there is a complex number $\lambda(m)$ for which
		\begin{align}\label{dhe}
			f(z)\mid T_m=\lambda(m)f(z).
		\end{align}
	\end{definition}
	We note the following theorem of Serre which will be useful to study the divisibility properties of $\bar{b}_{r,s}^k(n)$.
	\begin{theorem}[{{\cite[Theorem~2.65]{wom}}}]\label{6tdivi}
		Let $k$ and $m$ be positive integers. If $f(z)\in M_k(\Gamma_0(N),\chi(\bullet))$ has the Fourier expansion $f(z)=\sum_{n=0}^{\infty}a(n)q^n\in\mathbb{Z}[[q]]$, then there is a constant $\alpha>0$ such that
		\begin{align*}
			\#\{n\le X: c(n)\not\equiv 0\pmod{m}\}=O\left(\dfrac{X}{(log X)^\alpha}\right).
		\end{align*}
	\end{theorem}
	We also recall a result due to Cotron, Michaelsen, Stamm, and Zhu \cite{cotron}. We define
	\begin{align}\label{6gtau}
		G(\tau):=\frac{\eta(\delta_1\tau)^{r_1}\eta(\delta_2\tau)^{r_2}\cdots\eta(\delta_u\tau)^{r_u}}{\eta(\gamma_1\tau)^{s_1}\eta(\gamma_2\tau)^{s_2}\cdots\eta(\gamma_t\tau)^{s_t}},
	\end{align}
	where $r_i$, $s_i$, $\delta_i$, and $\gamma_i$ are positive integers with $\delta_1$, \dots, $\delta_u$, $\gamma_1$, \dots, $\gamma_t$ distinct and $u,t$$\ge$ 0. The weight of $G(\tau)$ is given by
	\begin{align*}
		\frac{1}{2}\left(\sum_{i=1}^{u}r_i -\sum_{i=1}^{t}s_i\right).
	\end{align*}
	Also define $D_G:=gcd(\delta_1, \delta_2, \dots, \delta_u)$. Then, we have the following result.
	\begin{theorem}[{{\cite[Theorem~1.1]{cotron}}}]\label{6tlacu2}
		Suppose $G(\tau)$ is an eta-quotient of the form \eqref{6gtau} with integer weight. If $p$ is a prime such that $p^a$ divides $D_G$ and 
		\begin{align*}
			p^a\ge \sqrt{\frac{\sum_{i=1}^{t}\gamma_is_i}{\sum_{i=1}^{u}\frac{r_i}{\delta_i}}},
		\end{align*}
		then $G(\tau)$ is lacunary modulo $p^j$ for any positive integer $j$. Moreover, there exists a positive constant $\alpha$, depending on $p$ and $j$, such that the number of integers $n\le X$ with $p^j$ not dividing $b(n)$ is $O\left(\dfrac{X}{log^\alpha X}\right)$. 
	\end{theorem}
	Ono and Taguchi \cite{taguchi} proved that the action of Hecke algebras on the spaces of modular forms of higher levels modulo $2$ is locally nilpotent. We require the following result of Ono and Taguchi for our later proof.
	\begin{theorem}[{{\cite[Theorem~1.3]{taguchi}}}]\label{6ttagu}
		Let $a$ and $n$ be a non-negative integers and let $\ell$ be a positive integer. Let $\chi$ be a Dirichlet character with conductor $2^n\cdot N$, where $N=1,3,5,15,17$. Then there is an integer $c\ge 0$ such that for every $f(z)\in M_\ell(\Gamma_0(2^aN),\chi)\cap \mathbb{Z}[[q]]$ and every $t\ge1$, we have
		\begin{align*}
			f(z)\mid T_{p_1}\mid T_{p_2}\mid \cdots T_{p_{c+t}}\equiv 0\pmod{2^t},
		\end{align*}
		whenever $p_1, p_2, \dots, p_{c+t}$ are odd primes not dividing $N$.
	\end{theorem}
	We also recall the weight $k$ Eisenstein series. For a positive integer $k$, let $\sigma_{k-1}(n)$ be the divisor function
	\begin{align*}
		\sigma_{k-1}(n):=\sum_{1\le d\mid n}^{}d^{k-1},
	\end{align*}
	and define the Bernoulli numbers $B_k$ as the coefficients of the series
	\begin{align*}
		\sum_{k=0}^{\infty}B_k\cdot \dfrac{t^k}{k!}=\dfrac{t}{e^t-1}=1-\dfrac{1}{2}t+\dfrac{1}{12}t^2-\dots.
	\end{align*}
	\begin{definition}[{{\cite[Definition~1.18]{wom}}}]
		If $k\ge2$ is even, then the weight $k$ Eisenstein series $E_k(z)$ is given by
		\begin{align*}
			E_k(z):=1-\dfrac{2k}{B_k}\sum_{n=1}^{\infty}\sigma_{k-1}(n)q^n.
		\end{align*}
	\end{definition}
	\begin{proposition}[{{\cite[Prop. 1.19 and Lemma 1.22]{wom}}}]\label{7p1}
		We have\\
		1. Suppose $k$ is even. If $p$ is prime and $(p-1)\mid k$, then $E_k(z)\equiv 1\pmod{p^{ord_p(2k)+1}}$.\\
		2. If $k\ge 4$ is even, then $E_k(z)\in M_k(\Gamma_0(1))$.
	\end{proposition}
	
	We now state a result of Sturm \cite{sturm}, which allows one to reduce the proof of a congruence to a finite computation. Let $m$ be a positive integer and $f(z)=\sum_{n=0}^{\infty}a(n)q^n$ with integer coefficients. The $ord_m
	(f(z))$ is defined by  $ord_m
	(f(z))=inf\{n\mid c(n)\not \equiv0 \pmod m\}$.
	\begin{theorem}[{{\cite[Theorem~2.58]{wom}}}]\label{sturm}
		Suppose p be a prime and $f(z)=\sum_{n=0}^{\infty}a(n)q^n$ and $g(z)=\sum_{n=0}^{\infty}b(n)q^n$ such that $f(z)$, $g(z) \in M_k(\Gamma_0(N), \chi)\cap \mathbb{Z}[[q]]$. If 
		\begin{align*}
			ord_p(f(z)-g(z))>\dfrac{kN}{12}\prod_{t\;prime, \;t\mid N}\bigg(\frac{1+t}{t}\bigg),
		\end{align*}
		then $ord_p(f(z)-g(z))=\infty$, $($which implies $a(n)\equiv b(n)\pmod p$$)$. 
	\end{theorem}
	We end this section with a recollection of the $U(d)$ operator. For a positive integer $d$, action of the operator $U(d)$ on a formal power series $\sum_{n=0}^{\infty}a(n)q^n$ is defined by
	\begin{align*}
		\sum_{n=0}^{\infty}a(n)q^n|U(d)=\sum_{n=0}^{\infty}a(dn)q^n.
	\end{align*}
	\begin{proposition}[{{\cite[Proposition 2.22]{wom}}}]
		Suppose that $f(z)\in M_k\left(\Gamma_0(N),\chi\right)$. If $d\mid N$, then
		\begin{align*}
			f(z)\mid U(d) \in M_k\left(\Gamma_0(N),\chi\right).
		\end{align*}
	\end{proposition}
	We note that if $d\nmid N$, then $f(z)$ is viewed as an element of $M_k\left(\Gamma_0(Nd),\chi\right)$. The $U(d)$ operator has the property that
	\begin{align*}
		\left[\left(\sum_{n=0}^{\infty}a(n)q^{n}\right)\sum_{n=0}^{\infty}b(n)q^{dn}\right]|U(d)= \left(\sum_{n=0}^{\infty}a(dn)q^n\right)\left(\sum_{n=0}^{\infty}b(n)q^n\right).
	\end{align*}
	We are now ready to prove our results listed in the previous section.
	\section{Proofs of Theorem \ref{6tmod4} and \ref{6tmod8}}\label{6s3} In this section, we first observe the functional equation for $\bar{B}^k_{r,s}(q)$ which is a direct generalization of Theorem \ref{6t1} and \cite[Corollary 3.2]{chacon}.
	\begin{lemma}\label{6l2}
		For all $r,s,k\ge1$, we have
		\begin{align*}
			\bar{B}^k_{r,s}(q)=\phi^{sk}(q)\prod_{i=1}^{\infty}\left(\phi(q^{2^i})\right)^{(r+s)k\cdot2^{i-1}}.
		\end{align*}	
	\end{lemma}
	\begin{proof}[\textbf{Proof of Theorem \ref{6tmod4}}]
		Using definition of $\phi(q)$ and Lemma \ref{6l2}, we have
		\begin{align*}
			\bar{B}^k_{r,s}(q)=\sum_{n=0}^{\infty}\bar{b}^k_{r,s}(n)q^n&=\phi^{sk}(q)\left(\phi(q^2)\right)^{(r+s)k}\prod_{i\ge2}^{\infty}\left(\phi(q^{2^i})\right)^{(r+s)k\cdot2^{i-1}}\\
			&\equiv\phi^{sk}(q)\left(\phi(q^2)\right)^{(r+s)k}\pmod{4}\\
			&\equiv \left(1+2\sum_{n=1}^{\infty}q^{n^2}\right)^{sk}\left(1+2\sum_{n=1}^{\infty}q^{2n^2}\right)^{(r+s)k}\pmod{4}.
		\end{align*}
		For $k=2j$ and $2j+1$, where $j\in\mathbb{Z}$, the above congruence becomes
		\begin{align*}
			\sum_{n\ge0}^{}\bar{b}^{2j}_{r,s}(n)q^n&\equiv \left(1+2\sum_{n=1}^{\infty}q^{n^2}\right)^{2js}\left(1+2\sum_{n=1}^{\infty}q^{2n^2}\right)^{2j(r+s)}\pmod{4}
		\end{align*}
		and
		\begin{align*}
			\sum_{n\ge0}^{}\bar{b}^{2j+1}_{r,s}(n)q^n&\equiv \left(1+2\sum_{n=1}^{\infty}q^{n^2}\right)^{(2j+1)s}\left(1+2\sum_{n=1}^{\infty}q^{2n^2}\right)^{(2j+1)(r+s)}\pmod{4},
		\end{align*} 
		respectively. Which further on simplification gives
		\begin{align*}
			\sum_{n=0}^{\infty}\bar{b}^{2j}_{r,s}(n)q^n&\equiv 1\pmod{4}
		\end{align*}
		and 
		\begin{align*}
			\sum_{n=0}^{\infty}\bar{b}^{2j+1}_{r,s}(n)q^n&\equiv \left(1+2\sum_{n=1}^{\infty}q^{n^2}\right)^s\left(1+2\sum_{n=1}^{\infty}q^{2n^2}\right)^{(r+s)}\pmod{4},
		\end{align*}
		respectively. 	This completes the proof of Theorem \ref{6tmod4}.
	\end{proof}
	
	\begin{proof}[\textbf{Proof of Theorem \ref{6tmod8}}]
		Using the definition of $\phi(q)$, Lemma \ref{6l2} and the congruence $\left(\phi(q^{2^i})\right)^{(r+s)k\cdot2^{i-1}}\equiv 1\pmod{8}$, for $i\ge3$, we have
		\begin{align*}
			\sum_{n=0}^{\infty}\bar{b}^k_{r,s}(n)q^n
			&\equiv\phi^{sk}(q)\left(\phi(q^2)\right)^{(r+s)k}\left(\phi(q^4)\right)^{(r+s)2k}\pmod{8}\\
			&\equiv \left(1+2sk\sum_{n=1}^{\infty}q^{n^2}+2sk(sk-1)\sum_{n=1}^{\infty}q^{2n^2}\right)\left(1+2(r+s)k\sum_{n=1}^{\infty}q^{2n^2}\right.\\
			&\quad\left. +2k(r+s)((r+s)k-1)\sum_{n= 1}^{\infty}q^{4n^2}\right)\left(1+4k(r+s)\sum_{n=1}^{\infty}q^{4n^2}\right.\\
			&\quad \left.+ 4k(r+s)(2k(r+s)-1)\sum_{n=1}^{\infty}q^{8n^2}\right)\pmod{8}\\
			&\equiv1+2sk\sum_{n=1}^{\infty}q^{n^2}+2k(s^2k+r)\sum_{n=1}^{\infty}q^{2(2n-1)^2}+2k(r+s)\left((r+s)k+1\right)\sum_{n=1}^{\infty}q^{4n^2}\\
			&\quad+4sk^2(r+s)(sk-1)\sum_{n=1}^{\infty}q^{4n^2}+4sk^2(r+s)\sum_{m,n=1}^{\infty}q^{n^2+2m^2}\\
			&\quad+4sk^2(r+s)\left((r+s)k-1\right)\sum_{m,n=1}^{\infty}q^{n^2+4m^2}\\
			&\quad+4sk^2(r+s)\left((r+s)k-1\right)(sk-1)\sum_{m,n=1}^{\infty}q^{2n^2+4m^2}\\
			&\quad+2k(3r+s^2k+2s)\sum_{n=1}^{\infty}q^{2(2n)^2}\pmod{8}.
		\end{align*}
		For $k=2j$ and $2j+1$, where $j\in\mathbb{Z}$, the above congruence becomes
		\begin{align*}
			\sum_{n=0}^{\infty}\bar{b}^k_{r,s}(n)q^n&\equiv1+2sk\sum_{n=1}^{\infty}q^{n^2}+2kr\sum_{n=1}^{\infty}q^{2(2n-1)^2}+2k(r+s)\sum_{n=1}^{\infty}q^{4n^2}\\
			&\quad+6rk\sum_{n=1}^{\infty}q^{2(2n)^2}\pmod{8}
		\end{align*}
		and
		\begin{align*}
			\sum_{n=0}^{\infty}\bar{b}^k_{r,s}(n)q^n&\equiv1+2sk\sum_{n=1}^{\infty}q^{n^2}+2k(s^2k+r)\sum_{n=1}^{\infty}q^{2(2n-1)^2}+2k(r+s)\left((r+s)k+1\right)\sum_{n=1}^{\infty}q^{4n^2}\\
			&\quad+4sk^2(r+s)\sum_{m,n=1}^{\infty}q^{n^2+2m^2}+2k(3r+s^2k+2s)\sum_{n=1}^{\infty}q^{2(2n)^2}\pmod{8},
		\end{align*}
		respectively. Which further on simplification gives
		\begin{align*}
			\sum_{n=0}^{\infty}\bar{b}^k_{r,s}(n)q^n&\equiv 1+2rk\sum_{n=1}^{\infty}q^{(2n)^2}+2sk\sum_{n=1}^{\infty}q^{(2n-1)^2}+2rk\sum_{n=1}^{\infty}q^{2(2n-1)^2}\\
			&\quad+6rk\sum_{n=1}^{\infty}q^{2(2n)^2}\pmod{8}
		\end{align*}
		and
		\begin{align*}
			\sum_{n=0}^{\infty}\bar{b}^k_{r,s}(n)q^n&\equiv 1+2sk\sum_{n=1}^{\infty}q^{(2n-1)^2}+2k\left(s+(r+s)\left((r+s)k+1\right)\right)\sum_{n=1}^{\infty}q^{(2n)^2}\\
			&\quad+2k(s^2k+r)\sum_{n=1}^{\infty}q^{2(2n-1)^2}+2k(3r+s^2k+2s)\sum_{n=1}^{\infty}q^{2(2n)^2}\\
			&\quad+4s(r+s)\sum_{m,n=1}^{\infty}q^{n^2+2m^2}\pmod{8},
		\end{align*}
		respectively. This completes the proof of Theorem \ref{6tmod8}.
	\end{proof}
	\section{Proofs of Theorems \ref{6tlac1}--\ref{6thec1}}\label{6s4}
	In this section, we prove Theorem \ref{6tlac1}-\ref{6thec1} using theory of modular forms.
	\begin{proof}[\textbf{Proof of Theorem \ref{6tlac1}}]
		From \eqref{6gf}, we have
		\begin{align*}
			\sum_{n= 0}^{\infty}\bar{b}^k_{r,s}(n)q^n=\dfrac{f_2^{(3s-2r)k}}{f_1^{2sk}f_4^{(s-r)k}}=\frac{\eta^{(3s-2r)k}(2z)}{\eta^{2sk}(z)\eta^{(s-r)k}(z)}.
		\end{align*}
		Using the notations in \eqref{6gtau} and its subsequent paragraph, we have
		\begin{align*}
			\delta_1=2,\;\; r_1=(3s-2r)k,\;\; \gamma_1=1,\;\; s_1=2sk,\;\; \gamma_2=4,\; \text{ and } \;s_2=(s-r)k.
		\end{align*}
		Also, $D_G=2$ and the weight is $-(rk)/2\in\mathbb{Z}$ if and only if $r=2^\alpha m$ or $k=2^\alpha m$ with $\alpha\ge 1$ and $gcd(2,m)=1$. We see that $2\mid D_G$ and 
		\begin{align*}
			2\ge\sqrt{\frac{2sk+4(s-r)k}{\dfrac{(3s-2r)k}{2}}}=\sqrt{4}=2.
		\end{align*}
		The proof follows choosing $p=2$ and $a=1$ in Theorem \ref{6tlacu2}.
	\end{proof}
	
	\begin{proof}[\textbf{Proof of Theorem \ref{6tdiv1}}]
		From \eqref{6gf}, we have
		\begin{align}\label{6egf}
			\sum_{n= 0}^{\infty}\bar{b}^k_{r,s}(n)q^n=\dfrac{f_2^{(3s-2r)k}}{f_1^{2sk}f_4^{(s-r)k}}.
		\end{align}
		For prime $p\ge 3$, we define
		\begin{align*}
			C(z):=\dfrac{\eta^p(24z)}{\eta(24pz)}.
		\end{align*}
		Using Lemma \ref{bt}, we get
		\begin{align}\label{6edivimodp}
			C(z)^{p^j}(z)=\dfrac{\eta^{p^{j+1}}(24z)}{\eta^{p^j}(24pz)}\equiv 1\pmod{p^{j+1}}.
		\end{align}
		We now define 
		\begin{align*}
			B_{r,s,k,j}(z):=\dfrac{\eta^{(3s-2r)k}(48z)}{\eta^{2sk}(24z)\eta^{(s-r)k}(96z)}\cdot C^{p^j}(z)=\dfrac{\eta^{p^{j+1}-2sk}(24z)\eta^{(3s-2r)k}(48z)}{\eta^{p^j}(24pz)\eta^{(s-r)k}(96z)}.
		\end{align*}
		From \eqref{6egf} and \eqref{6edivimodp}, we have
		\begin{align}\label{6ediv2}
			B_{r,s,k,j}(z)\equiv\dfrac{\eta^{(3s-2r)k}(48z)}{\eta^{2sk}(24z)\eta^{(s-r)k}(96z)}= \dfrac{f_{48}^{(3s-2r)k}}{f_{24}^{2sk}f_{96}^{(s-r)k}}\equiv \sum_{n= 0}^{\infty}\bar{b}^k_{r,s}(n)q^{24n}\pmod{p^{j+1}}.
		\end{align}
		Next, we show that $B_{r,s,k,j}(z)$ is a modular form. By Theorem \ref{t2.1}, we see that the level of $B_{r,s,k,j}(z)$ is $N=96M$, where $M$ is the smallest positive integer which satisfies
		\begin{align*}
			96M\left(\dfrac{p^{j+1}-2sk}{24}+\dfrac{(3s-2r)k}{48}-\dfrac{p^j}{24p}-\dfrac{(s-r)k}{96}\right)\equiv 0\pmod{24},
		\end{align*}
		which is equivalent to
		\begin{align*}
			M\left(4p^{j-1}(p^2-1)-3k(r+s)\right)\equiv 0\pmod{24}.
		\end{align*}
		Therefore, $M=2^3$ so that $N=2^83$. Since the cusps of $\Gamma_0(2^83)$ are of the form $c/d$ where $d\mid 2^83$ and $gcd(c,d)=1$, it follows from Theorem \ref{t2.2} that $B_{r,s,k,j}(z)$ is holomorphic at a cusp $c/d$ if and only if 
		\begin{align}\label{6ediv1}
			(p^{j+1}-2sk)\dfrac{gcd(d,24)^2}{24}+(3s-2r)k\dfrac{gcd(d,48)^2}{48}-p^j\dfrac{gcd(d,24p)^2}{24p}\nonumber\\
			-(s-r)k\dfrac{gcd(d,96)^2}{96}\ge 0.
		\end{align}
		We now examine the non-negativity of \eqref{6ediv1} for all possible divisors of $2^83$.\\
		For $d=2^a3^b$ with $0\le a\le 3$ and $0\le b\le1$, \eqref{6ediv1} becomes
		\begin{align}\label{6einq1}
			4p^{j-1}(p^2-1)-3k(r+s)\ge 0.
		\end{align}
		For $d=2^43^b$ with $0\le b\le1$, \eqref{6ediv1} becomes
		\begin{align}\label{6einq2}
			p^{j-1}(p^2-1)-3k(r-s)\ge 0.
		\end{align}
		For $d=2^a3^b$ with $5\le a\le 8$ and $0\le b\le1$, \eqref{6ediv1} becomes
		\begin{align*}
			p^{j-1}(p^2-1)\ge 0,
		\end{align*}
		which is always true. The inequalities \eqref{6einq1} and \eqref{6einq2} hold since we have assumed that $p^{j-1}(p^2-1)\ge 3k(r+s)$. Thus, $B_{r,s,k,j}(z)$ is holomorphic at cusp $c/d$. The weight of $B_{r,s,k,j}(z)$ is $\ell= \left(p^j(p-1)-rk\right)/2$, which is an integer since $r=2^\alpha m$ or $k=2^\alpha m$ with $\alpha\ge 1$ and $gcd(2,m)=1$. The associated character is given by 
		\begin{align*}
			\chi(\bullet)=\left(\dfrac{(-1)^\ell 2^{3p^j(p-1)+sk-3rk}3^{p^j(p-1)-rk}p^{-p^j}}{\bullet}\right).
		\end{align*}
		Hence, $B_{r,s,k,j}(z)\in M_{\left(p^j(p-1)-rk\right)/2}(\Gamma_0(2^83),\chi(\bullet))$.\\
		Therefore, by Theorem \ref{6tdivi}, the Fourier coefficients of $B_{r,s,k,j}(z)$ are almost always divisible by $p^j$. Combining this with \eqref{6ediv2}, we complete the proof for divisibility of $\bar{b}^k_{r,s}(n)$.
	\end{proof}
	
	\begin{proof}[\textbf{Proof of Theorem \ref{6tdiv2}}]
		From equation \eqref{6ediv2} with $p=2$, we have
		\begin{align*}
			B_{r,s,k,j}(z)\equiv \sum_{n= 0}^{\infty}\bar{b}^k_{r,s}(n)q^{24n}\pmod{2^{j+1}},
		\end{align*}
		which implies
		\begin{align}\label{6etag1}
			B_{r,s,k,j}(z):=\sum_{n=0}^{\infty}	B_{r,s,k,j}(n)q^n\equiv\sum_{n=0}^{\infty}\bar{b}^k_{r,s}\left(\dfrac{n}{24}\right)q^n\pmod{2^{j+1}}.
		\end{align}
		We have $B_{r,s,k,j}(z)\in M_{2^{j-1}-rk/2}(\Gamma_0(2^83), \chi(\bullet))$, where $\chi(\bullet)$ denotes the associated character. By Theorem \ref{6ttagu}, we find that there exits an integer $t\ge0$ such that for every $u\ge 1$ and primes $p_1,p_2, \dots, p_{t+u}$ greater than $3$, we have
		\begin{align*}
			B_{r,s,k,j}(z)\mid T_{p_1}\mid T_{p_2}\mid\cdots T_{t+u}\equiv 0\pmod{2^j}.
		\end{align*}
		Using the definition of Hecke operators, it follows that if $p_1,p_2, \dots, p_{t+u}$ are distinct primes and if $n\ge 1$ is coprime to $p_1,p_2, \dots, p_{t+u}$, then
		\begin{align}\label{6etag2}
			B_{r,s,k,j}(p_1\cdots p_{t+u}\cdot n)\equiv 0\pmod{2^j}.
		\end{align}
		In view of \eqref{6etag1} and \eqref{6etag2}, we complete the proof.
	\end{proof}
	We use the following lemma and Hecke eigen form theory to prove congruences modulo powers of $2$.
	\begin{lemma}\label{6lhec}
		For integers $n,\ell,k\ge0$, $r,\alpha, \beta\ge 1$ such that $\alpha$ and $\beta$ are odd, we have
		\begin{align*}
			\sum_{n=0}^{\infty}\bar{b}^{2^\ell\beta}_{r, 2^k\alpha}(2n+1)q^n	\equiv 2^{k+\ell+1}f_1^{12}\pmod{2^{k+\ell+2}}.
		\end{align*}
	\end{lemma}
	\begin{proof}
		Thanks to \eqref{6gf}, we have
		\begin{align*}
			\sum_{n=0}^{\infty}\bar{b}^{2^\ell\beta}_{r, 2^k\alpha}(n)q^n=\dfrac{f_2^{(3\cdot2^k\alpha-2r)2^\ell\beta}}{f_1^{2^{k+\ell+1}\alpha\beta}f_4^{(2^k\alpha-r)2^\ell\beta}}.
		\end{align*}
		Employing \eqref{edf2} in the above equation, we get
		\begin{align*}
			\sum_{n=0}^{\infty}\bar{b}^{2^\ell\beta}_{r, 2^k\alpha}(n)q^n&=\dfrac{f_2^{(3\cdot2^k\alpha-2r)2^\ell\beta}}{f_4^{(2^k\alpha-r)2^\ell\beta}}\left(\dfrac{f_8^5}{f_2^5f_{16}^2}+2q\dfrac{f_4^2f_{16}^2}{f_2^5f_8}\right)^{2^{k+\ell}\alpha\beta}\\
			&=\dfrac{f_2^{(3\cdot2^k\alpha-2r)2^\ell\beta}}{f_4^{(2^k\alpha-r)2^\ell\beta}}\left(\dfrac{f_8^5}{f_2^5f_{16}^2}\right)^{2^{k+\ell}\alpha\beta}\sum_{i=0}^{2^{k+\ell}\alpha\beta}\binom{2^{k+\ell}\alpha\beta}{i}2^iq^i \dfrac{f_4^{2i}f_{16}^{4i}}{f_8^{6i}}.
		\end{align*}
		Extracting the terms of the form $q^{2n+1}$ from the above equation, we obtain
		\begin{align*}
			\sum_{n=0}^{\infty}\bar{b}^{2^\ell\beta}_{r, 2^k\alpha}(2n+1)q^n&=\dfrac{f_1^{(3\cdot2^k\alpha-2r)2^\ell\beta}}{f_2^{(2^k\alpha-r)2^\ell\beta}}\left(\dfrac{f_4^5}{f_1^5f_{8}^2}\right)^{2^{k+\ell}\alpha\beta}\sum_{i=0}^{2^{k+\ell}\alpha\beta}\binom{2^{k+\ell}\alpha\beta}{2i+1}2^{2i+1}q^i \dfrac{f_2^{4i+2}f_{8}^{8i+4}}{f_4^{12i+6}}.
		\end{align*}
		Since $\alpha$ and $\beta$ are odd integers, and for $i\ge1$,
		\begin{align*}
			2^{2i+1}\binom{2^{k+\ell}\alpha\beta}{2i+1}\equiv 0\pmod{2^{k+\ell+2}}, 
		\end{align*}
		we have
		\begin{align*}
			\sum_{n=0}^{\infty}\bar{b}^{2^\ell\beta}_{r, 2^k\alpha}(2n+1)q^n&\equiv 2^{k+\ell+1}\dfrac{f_4^{5\cdot2^{k+\ell}\alpha\beta-6}f_8^{4-2^{k+\ell+1}\alpha\beta}}{f_1^{(2^k\alpha-r)2^{\ell+1}\beta}f_2^{(2^k\alpha-r)2^\ell\beta-2}}\pmod{2^{k+\ell+2}}.
		\end{align*}
		Applying Lemma \ref{bt} in the above congruence, yields the desired result.
	\end{proof}
	
	\begin{proof}[\textbf{Proof of Theorem \ref{6thec1}}]
		Thanks to Lemma \ref{bt} and \ref{6lhec}, we have
		\begin{align*}
			\sum_{n=0}^{\infty}\bar{b}^{2^\ell\beta}_{r,2^k\alpha}(2n+1)q^{2n+1}\equiv 2^{k+\ell+1}q f_4^6\pmod{2^{k+\ell+2}},
		\end{align*}
		which implies
		\begin{align}\label{ehec1}
			\sum_{n=0}^{\infty}\bar{b}^{2^\ell\beta}_{r,2^k\alpha}(2n+1)q^{2n+1}\equiv 2^{k+\ell+1}\eta^6(4z)\pmod{2^{k+\ell+2}}.
		\end{align}
		From Theorems \ref{t2.1} and \ref{t2.2}, it is easy to find that $\eta^6(4z)\in S_3\left(\Gamma_0(16), \left(\frac{-4^6}{.}\right)\right)$. Let $\sum_{n=1}^{\infty}c(n)q^n$ be the Fourier expansion of  $\eta^6(4z)$. Thus, we have
		\begin{align*}
			\sum_{n=1}^{\infty}c(n)q^n=q-6q^5+9q^9+10q^{13}-30q^{17}+\dots
		\end{align*} 
		and we note that $c(n)=0$, if $n\not\equiv 1\pmod 4$ for all $n\ge 0$.\\ 
		Comparing coefficients in \eqref{ehec1}, we obtain
		\begin{align}\label{ehec2}
			\bar{b}^{2^\ell\beta}_{r,2^k\alpha}(2n+1)\equiv 2^{k+\ell+1}c(2n+1)\pmod{2^{k+\ell+2}}.
		\end{align}
		Since  $\eta^6(4z)$ is a Hecke eigenform \cite{martin}, using \eqref{2.7} and \eqref{dhe}, we obtain
		\begin{align}\label{e1.4.4}
			\eta^6(4z)\mid T_p=\sum_{n=1}^{\infty}\left(c(pn)+p^2\left(\frac{-4^6}{p}\right)c\left(\frac{n}{p}\right)\right)q^n=\lambda(p)\sum_{n=1}^{\infty}c(n)q^n.
		\end{align}
		Comparing the coefficients in \eqref{e1.4.4}, we get
		\begin{align}\label{e1.4.5}
			c(pn)+p^2\left(\frac{-1}{p}\right)c\left(\frac{n}{p}\right)=\lambda(p)c(n).
		\end{align}
		Since $c(1)=1$, setting $n=1$ in \eqref{e1.4.5}, we obtain
		\begin{align*}
			c(p)=\lambda(p).
		\end{align*}
		Also since $c(p)=0$ for all $p\not\equiv 1\pmod 4$, the above equation  implies $\lambda(p)=0$, for all $p\not\equiv 1\pmod 4$.	Note that $\left(\frac{-1}{p}\right)=-1$ if $p\not\equiv 1\pmod 4$. Therefore, from \eqref{e1.4.5}, we have 
		\begin{align}\label{e1.4.7}
			c(pn)-p^2c\left(\frac{n}{p}\right)=0,
		\end{align}
		for all $p\not\equiv 1\pmod 4$.
		Replacing $n$ by $pn+m$ with $gcd(n,m)=1$ in \eqref{e1.4.7}, we see that for all $n\ge 0$ with $p\nmid n$,
		\begin{align}\label{e1.4.8}
			c(p^2n+pm)=0.
		\end{align}
		Replacing $n$ by $2n-pm+1$ in \eqref{e1.4.8} and employing the resulting equation in \eqref{ehec2}, we obtain
		\begin{align}\label{e1.4.10}
			\bar{b}_{r,2^k\alpha}^{2^\ell\beta}\left(2np^2+pm(1-p^2)+p^2\right)\equiv 0\pmod {2^{k+\ell+2}},
		\end{align}
		with $gcd(m,p)=1$.\\
		Since $p\equiv 3\pmod 4$, we have $4\mid(1-p^2)$ and $gcd(1-p^2,p)=1$. Thus, when $m$ runs over a residue system excluding the multiples of $p$, so does $m(1-p^2)$. Therefore, for $p\nmid t$, we can rewrite \eqref{e1.4.10} as 
		\begin{align}\label{e1.4.11}
			\bar{b}_{r,2^k\alpha}^{2^\ell\beta}(2p^2n+pt+p^2)\equiv 0\pmod {2^{k+\ell+2}}.
		\end{align}
		Replacing $n$ by $pn$ in \eqref{e1.4.7}, we obtain
		\begin{align}\label{e1.4.12}
			c(p^2n)=p^2c(n).
		\end{align}
		Now, replacing $n$ by $2n+1$ in \eqref{e1.4.12}, we get
		\begin{align}\label{e1.4.13}
			c(2p^2n+p^2)=p^2c(2n+1).
		\end{align}
		From \eqref{ehec2} and \eqref{e1.4.13}, we obtain
		\begin{align}
			\bar{b}_{r,2^k\alpha}^{2^\ell\beta}(2p^2n+p^2)\equiv p^2	\bar{b}_{r,2^k\alpha}^{2^\ell\beta}(2n+1)\pmod{2^{k+\ell+2}}.
		\end{align}
		Let $p_i\ge 3$ be primes such that $p_i\equiv 3\pmod 4$. Since
		\begin{align*}
			2p_1^2p_2^2\cdots p_k^2 n+p_1^2p_2^2\cdots p_j^2=2p_1^2\left(p_2^2\cdots p_k^2 n+\dfrac{p_2^2\cdots p_j^2 -1}{4}\right)+p_1^2,
		\end{align*}
		using \eqref{e1.4.13} repeatedly and then employing \eqref{e1.4.11}, we obtain that
		\begin{align*}
			&\bar{b}_{r,2^k\alpha}^{2^\ell\beta}\left(2p_1^2p_2^2\cdots p_k^2p_{j+1}^2n+p_1^2p_2^2\cdots p_k^2p_{j+1}\left(t+p_{j+1}\right)\right)\\
			&\equiv p_1^2 \bar{b}_{r,2^k\alpha}^{2^\ell\beta}\left(2p_1^2p_2^2\cdots p_j^2p_{j+1}^2n+p_1^2p_2^2\cdots p_k^2p_{j+1}\left(t+p_{j+1}\right)\right)\\
			\vdots\\
			&\equiv (p_1p_2\cdots p_j)^2\bar{b}_{r,2^k\alpha}^{2^\ell\beta}\left(2p_{j+1}^2n+p_{j+1}(t+p_{j+1})\right)\\
			&\equiv 0\pmod{2^{k+\ell+2}},
		\end{align*}
		when $t\not\equiv 0\pmod{p_{j+1}}$. This completes the proof.
	\end{proof}
		\section{Proof of Theorems \ref{7t1.1}--\ref{7t1.3}}
	The following lemmas compute the Sturm's bound from which our congruences stated in Theorems \ref{7t2}-\ref{7t3} can be verified directly. In order to state our lemmas, we define
	\begin{align}
		N&:=\begin{cases}
			2^{4-t}, & \text{if } 2^t\mid rk,\\[4pt]
			2^4, & \text{ otherwise}.
		\end{cases} \label{n1}\\
		N&:=\begin{cases}
			2^{5-t}, & \text{if } 2^t\mid k(r+s),\\[4pt]
			2^5, & \text{ otherwise}.
		\end{cases}	\label{n2}
	\end{align}
	where $t=1,2,3$.
	\begin{lemma}\label{7t2}
		Let $r,s,k\ge1$ with $rk\equiv 0\pmod{2}$ and let $p\ge3$ be any prime such that $rk<12p$ and $8p\ge k(r+s)$. Suppose that $\bar{b}^k_{r,s}(2pn+p)\equiv 0\pmod{p}$ for all $n\le B:=\dfrac{(12p-rk)}{8}\cdot3N$, where $N$ is the smallest positive integer such that  when $r=s$ and $r\not =s$, $N$ is defined as in \eqref{n1} with $t\not=3$ and \eqref{n2}, respectively.Then $\bar{b}^k_{r,s}(2pn+p)\equiv 0\pmod{p}$ for all $n\ge 0$.
	\end{lemma}
	\begin{proof}
		For prime $p\ge 3$, we define 
		\begin{align*}
			G_{r,s,k}(z):=\dfrac{\eta^{(3s-2r)k+12p}(2z)}{\eta^{2sk}(z)\eta^{(s-r)k}(4z)}.
		\end{align*}
		Thanks to \eqref{6gf}, we have
		\begin{align}\label{7e3}
			G_{r,s,k}(z)\equiv \left(\sum_{n=0}^{\infty}\bar{b}^k_{r,s}(n)q^{n+p}\right)\cdot f_{2p}^{12}\pmod{p}.
		\end{align}
		First, to show $G_{r,s,k}(z)$ is a modular form, we observe from Theorem \ref{t2.1} that the level of $G_{r,s,k}(z)$ is the smallest positive integer $N$ such that 
		\begin{align*}
			N\left(\dfrac{(3s-2r)k+12p}{2}-2sk-\dfrac{(s-r)k}{4}\right)\equiv 0\pmod{24}.
		\end{align*}
		For $N$ as in \eqref{n1} and \eqref{n2}, the above identity is satisfied. It follows from Theorem \ref{t2.2} that $G_{r,s,k}(z)$ is holomorphic at a cusp $c/d$ if and only if
		\begin{align}\label{7e4}
			\left((3s-2r)k+12p\right)\cdot \dfrac{gcd(d,2)^2}{2}-2sk-(s-r)k\cdot \dfrac{gcd(d,4)^2}{4}\ge 0.
		\end{align}
		We now examine the positivity of \eqref{7e4} for all possible divisors of $N$. For $d=1$ and $d=2$, \eqref{7e4} implies $	8p\ge k(r+s)$ and $8p\ge k(r-s)$, respectively,  which are true since we have assumed $8p\ge k(r+s)$.
		For $d=4,8,16,32$, \eqref{7e4} implies $24p\ge 0$,	which is always true. Thus, $G_{r,s,k}(z)$ is holomorphic at cusps $c/d$. The weight of $F_{r,s,k}(z)$ is $\ell=(12p-rk)/2$, which is a positive integer since $rk\equiv 0\pmod{2}$ and $rk<12p$. The associated character is given by $\chi(\bullet)=\left(\frac{(-1)^\ell 2^{sk+12p}}{\bullet}\right)$.
		Hence, $G_{r,s,k}(z)\in M_{(12p-rk)/2}\left(\Gamma_0(N), \chi(\bullet)\right)$.\\
		Now, on applying $U(p)$ operator to \eqref{7e3}, we obtain
		\begin{align*}
			G_{r,s,k}(z)\mid U(p)&\equiv \left(\sum_{n=0}^{\infty}\bar{b}^k_{r,s}(n)q^{n+p}\cdot f_{2p}^{12}\right) \mid  U(p)\pmod{p}\\
			&=\left(\sum_{n\ge 1}^{\infty}\bar{b}^k_{r,s}(pn-p)q^n\right)\cdot f_2^{12}\\
			&=\left(\sum_{n=0}^{\infty}\bar{b}^k_{r,s}(pn)q^{n+1}\right)\cdot f_2^{12}.
		\end{align*}
		Similarly, on applying $U(2)$ operator to the above congruence, we obtain
		\begin{align*}
			\left(G_{r,s,k}(z)\mid U(p)\right)\mid U(2)&\equiv \left(\sum_{n=0}^{\infty}\bar{b}^k_{r,s}(pn)q^{n+1}\cdot f_{2}^{12}\right)\mid U(2) \pmod{p}\\
			&=\left(\sum_{n\ge1}^{\infty}\bar{b}^k_{r,s}(2pn-p)q^n\right)\cdot f_1^{12}\\
			&=\left(\sum_{n=0}^{\infty}\bar{b}_{r,s}^k(2pn+p)q^{n+1}\right)\cdot f_1^{12}.
		\end{align*}
		The Sturm's bound for the associated space of modular forms is found to be  $ B:=\left(\dfrac{12p-rk}{24}\right)5N \displaystyle\prod_{t\;prime, \;t\mid 5N}^{}\dfrac{t+1}{t}=\dfrac{(12p-rk)}{8}\cdot3N$. Therefore, using Theorem \ref{sturm}, we complete the proof of Theorem \ref{7t2}.	
	\end{proof}
	\begin{lemma}\label{7t5}
		Let $r,s,k\ge1$ with $rk\equiv 1\pmod{2}$ and let $p\ge3$ be any prime such that $rk<11p$ and $7p\ge k(r+s)$. Suppose that $\bar{b}^k_{r,s}(2pn+p)\equiv 0\pmod{p}$ for all $n\le B:=12\cdot(11p-rk)$.
		Then $\bar{b}^k_{r,s}(2pn+p)\equiv 0\pmod{p}$ for all $n\ge 0$.
	\end{lemma}
	\begin{proof}
		For prime $p\ge 3$, we define 
		\begin{align*}
			K_{r,s,k}(z):=\dfrac{\eta^{(3s-2r)k+10p}(2z)\eta^{p-(s-r)k}(4z)}{\eta^{2sk}(z)}.
		\end{align*}
		Thanks to \eqref{6gf}, we have
		\begin{align}\label{7e12}
			K_{r,s,k}(z)\equiv \left(\sum_{n=0}^{\infty}\bar{b}^k_{r,s}(n)q^{n+p}\right)\cdot f_{2p}^{10}f_{4p}\pmod{p}.
		\end{align}
		By Theorem \ref{t2.1}, we see that the level of $K_{r,s,k}(z)$ is the smallest positive integer $N$ such that 
		\begin{align*}
			N\left(\dfrac{(3s-2r)k+10p}{2}-2sk- \dfrac{(s-r)k+p}{4}\right)\equiv 0\pmod{24}.
		\end{align*}
		Thus, we have $N=32$. It follows from Theorem \ref{t2.2} that $K_{r,s,k}(z)$ is holomorphic at a cusp $c/d$ if and only if
		\begin{align}\label{7e13}
			\left((3s-2r)k+10p\right)\dfrac{gcd(d,2)^2}{2}+(p-(s-r)k)\dfrac{gcd(d,4)^2}{4}-2sk\ge 0.
		\end{align}
		For $d=1$ and $d=2$, \eqref{7e13} implies $	7p\ge k(r+s)$ and $7p\ge k(r-s)$, respectively, which are true since we have assumed $7p\ge k(r+s)$.
		For $d=4,8,16,32$, \eqref{7e13} implies $24p\ge 0$,	which is always true. Thus, $K_{r,s,k}(z)$ is holomorphic at cusp $c/d$. The weight of $K_{r,s,k}(z)$ is $\ell=(11p-rk)/2$, which is a positive integer since $rk\equiv 1\pmod{2}$ and $rk<11p$. The associated character is given by $\chi(\bullet)=\left(\frac{(-1)^\ell 2^{sk+12p}}{\bullet}\right)$.
		Hence, $K_{r,s,k}(z)\in M_{(11p-rk)/2}\left(\Gamma_0(32), \chi(\bullet)\right)$.\\
		Now, on applying $U(p)$ operator to \eqref{7e12}, we obtain
		\begin{align*}
			K_{r,s,k}(z)\mid U(p)&\equiv \left(\sum_{n=0}^{\infty}\bar{b}^k_{r,s}(n)q^{n+p}\cdot f_{2p}^{10}f_{4p}\right) \mid  U(p)\pmod{p}\\
			&=\left(\sum_{n=0}^{\infty}\bar{b}^k_{r,s}(pn)q^{n+1}\right)\cdot f_2^{10}f_4.
		\end{align*}
		Similarly, on applying $U(2)$ operator to the above congruence, we obtain
		\begin{align*}
			\left(K_{r,s,k}(z)\mid U(p)\right)\mid U(2)&\equiv \left(\sum_{n=0}^{\infty}\bar{b}^k_{r,s}(pn)q^{n+1}\cdot f_2^{10}f_4\right)\mid U(2) \pmod{p}\\
			&=\left(\sum_{n=0}^{\infty}\bar{b}_{r,s}^k(2pn+p)q^{n+1}\right)\cdot f_1^{10}f_2.
		\end{align*}
		The Sturm's bound for the associated space of modular forms is found to be  $ B:=\left(\dfrac{11p-rk}{3}\right)20 \displaystyle\prod_{t\;prime, \;t\mid 160}^{}\dfrac{t+1}{t}=12\cdot(11p-rk)$. Therefore, using Theorem \ref{sturm}, we complete the proof of Theorem \ref{7t5}.	
	\end{proof}
	\begin{lemma}\label{7t1}
		Let $r,s,k\ge1$ with $rk\equiv 0\pmod{2}$ and let $p\ge3$ be any prime such that $rk<24p$ and $8p\ge k(r+s)$. Suppose that $\bar{b}^k_{r,s}(pn+p)\equiv 0\pmod{p}$ for all $n\le B:=\dfrac{(24p-rk)}{16}\cdot N$, where $N$ is the smallest positive integer such that when $r=s$ and $r\not =s$, $N$ is defined as in \eqref{n1} and \eqref{n2}, respectively. Then $\bar{b}^k_{r,s}(pn)\equiv 0\pmod{p}$ for all $n\ge 1$.
	\end{lemma}
	\begin{proof}
		For prime $p\ge 3$, we define 
		\begin{align*}
			F_{r,s,k}(z):=\dfrac{\eta^{(3s-2r)k}(2z)\eta^{24p-2sk}(z)}{\eta^{(s-r)k}(4z)}.
		\end{align*}
		Thanks to \eqref{6gf}, we have
		\begin{align}\label{7e2}
			F_{r,s,k}(z)\equiv \left(\sum_{n=0}^{\infty}\bar{b}^k_{r,s}(n)q^{n+p}\right)\cdot f_p^{24}\pmod{p}.
		\end{align}
		By Theorem \ref{t2.1}, we see that the level of $F_{r,s,k}(z)$ is the smallest positive integer $N$ such that 
		\begin{align*}
			N\left(\dfrac{(3s-2r)k}{2}+24p-2sk-\dfrac{(s-r)k}{4}\right)\equiv 0\pmod{24}.
		\end{align*}
		For $N$ as in \eqref{n1} and \eqref{n2}, the above identity is satisfied. It follows from Theorem \ref{t2.2} that $F_{r,s,k}(z)$ is holomorphic at a cusp $c/d$ if and only if
		\begin{align}\label{7e1}
			(3s-2r)k\cdot\dfrac{gcd(d,2)^2}{2}+24p-2sk-(s-r)k\cdot\dfrac{gcd(d,4)^2}{4}\ge 0.
		\end{align}
		For $d=1$ and $d=2$, \eqref{7e1} implies $	32p\ge k(r+s)$ and $8p\ge k(r-s)$, respectively, which are true since we have assumed $8p\ge k(r+s)$.
		For $d=4,8,16,32$, \eqref{7e1} implies $24p\ge 0$,	which is always true. Thus, $F_{r,s,k}(z)$ is holomorphic at cusp $c/d$. The weight of $F_{r,s,k}(z)$ is $\ell=(24p-rk)/2$, which is a positive integer since $rk\equiv 0\pmod{2}$ and $rk<24p$. The associated character is given by $\chi(\bullet)=\left(\frac{(-1)^\ell 2^{sk}}{\bullet}\right)$.
		Hence, $F_{r,s,k}(z)\in M_{(24p-rk)/2}\left(\Gamma_0(N), \chi(\bullet)\right)$.\\
		Now, on applying $T_p$ operator to \eqref{7e2}, we obtain
		\begin{align*}
			F_{r,s,k}(z)\mid T_p\equiv \left(\sum_{n=0}^{\infty}\bar{b}^k_{r,s}(pn+p)q^{n+2}\right)\cdot f_1^{24}\pmod{p}.
		\end{align*}
		The Sturm's bound for the associated space of modular forms is found to be  $ B:=\left(p-\dfrac{rk}{24}\right)N \displaystyle\prod_{t\;prime, \;t\mid N}^{}\dfrac{t+1}{t}=\dfrac{(24p-rk)}{16}\cdot N$. Therefore, using Theorem \ref{sturm}, we complete the proof of Theorem \ref{7t1}.
	\end{proof}
	\begin{lemma}\label{7t4}
		Let $r,s,k\ge1$ with $rk\equiv 1\pmod{2}$ and let $p\ge3$ be any prime such that $rk<23p$ and $8p\ge k(r+s)$. Suppose that $\bar{b}^k_{r,s}(pn+p)\equiv 0\pmod{p}$ for all $n\le B:=\dfrac{(23p-rk)}{16}\cdot N$, where $N$ is the smallest positive integer such that\\
		1. when $r=s$, 	$N=8$\\
		2. when $r\not =s$, 	$	N=
		\begin{cases}
			16, & \text{if } 2\mid k(r+s),\\[4pt]
			32, & \text{otherwise }.
		\end{cases}$\\
		Then $\bar{b}^k_{r,s}(pn)\equiv 0\pmod{p}$ for all $n\ge 1$.
	\end{lemma}
	\begin{proof}
		For prime $p\ge 3$, we define 
		\begin{align*}
			J_{r,s,k}(z):=\dfrac{\eta^{22p-2sk}(z)\eta^{(3s-2r)k+p}(2z)}{\eta^{(s-r)k}(4z)}.
		\end{align*}
		Thanks to \eqref{6gf}, we have
		\begin{align}\label{7e10}
			J_{r,s,k}(z)\equiv \left(\sum_{n=0}^{\infty}\bar{b}^k_{r,s}(n)q^{n+p}\right)\cdot f_p^{2}f_{2p}^{11}\pmod{p}.
		\end{align}
		By Theorem \ref{t2.1}, we see that the level of $J_{r,s,k}(z)$ is the smallest positive integer $N$ such that 
		\begin{align*}
			N\left(22p-2sk +\dfrac{(3s-2r)k+p}{2}-\dfrac{(s-r)k}{4}\right)\equiv 0\pmod{24},
		\end{align*}
		which holds for the choice of $N$ as in  the hypothesis. Therefore, it follows from Theorem \ref{t2.2} that $J_{r,s,k}(z)$ is holomorphic at a cusp $c/d$ if and only if
		\begin{align}\label{7e11}
			22p-2sk+\left((3s-2r)k+p\right)\dfrac{gcd(d,2)^2}{2}-(s-r)k\dfrac{gcd(d,4)^2}{4}\ge 0.
		\end{align}
		For $d=1$ and $d=2$, \eqref{7e11} implies $	30p\ge k(r+s)$ and $8p\ge k(r-s)$, respectively, which are true since we have assumed $8p\ge k(r+s)$.
		For $d=4,8,16,32$, \eqref{7e1} implies $24p\ge 0$,	which is always true. Thus, $J_{r,s,k}(z)$ is holomorphic at cusp $c/d$. The weight of $J_{r,s,k}(z)$ is $\ell=(23p-rk)/2$, which is a positive integer since $rk\equiv 1\pmod{2}$ and $rk<23p$. The associated character is given by $\chi(\bullet)=\left(\frac{(-1)^\ell 2^{sk+p}}{\bullet}\right)$.
		Hence, $J_{r,s,k}(z)\in M_{(23p-rk)/2}\left(\Gamma_0(N), \chi(\bullet)\right)$.\\
		Now, on applying $T_p$ operator to \eqref{7e2}, we obtain
		\begin{align*}
			F_{r,s,k}(z)\mid T_p\equiv \left(\sum_{n=0}^{\infty}\bar{b}^k_{r,s}(pn+p)q^{n+2}\right)\cdot f_1^2f_2^{11}\pmod{p}.
		\end{align*}
		The Sturm's bound for the associated space of modular forms is found to be  $B:=\left(\dfrac{23p-rk}{24}\right)N \displaystyle\prod_{t\;prime, \;t\mid N}^{}\dfrac{t+1}{t}=\dfrac{(23p-rk)}{16}\cdot N$. Therefore, using Theorem \ref{sturm}, we complete the proof of Theorem \ref{7t4}.
	\end{proof}
	\begin{lemma}\label{7t3}
		Let $r,s,k\ge1$ with $rk\equiv 0\pmod{2}$ and let $p\ge5$ be any prime such that $rk<24p^2$ and $8p\ge k(r+s)$. Suppose that $\bar{b}^k_{r,s}(p^2n)\equiv\bar{b}^k_{r,s}(pn) \pmod{p}$ for all $n\le B:=\dfrac{(24p^2-rk)}{16}\cdot N$, where $N$ is the smallest positive integer such that when $r=s$ and $r\not =s$, $N$ is defined as in \eqref{n1} and \eqref{n2}, respectively. Then $\bar{b}^k_{r,s}(p^2n)\equiv\bar{b}^k_{r,s}(pn) \pmod{p}$ for all $n\ge 0$.
	\end{lemma}
	\begin{proof}
		For prime $p\ge 5$, we define 
		\begin{align*}
			H_{r,s,k}(z):=\dfrac{\eta^{24p-2sk}(z)\eta^{(3s-2r)k}(2z)}{\eta^{(s-r)k}(4z)}\cdot E_{p-1}^{12p}(z)
		\end{align*}
		and
		\begin{align*}
			I_{r,s,k}(z):=\dfrac{\eta^{24p^2-2sk}(z)\eta^{(3s-2r)k}(2z)}{\eta^{(s-r)k}(4z)}.
		\end{align*}
		Thanks to \eqref{6gf} and Proposition \ref{7p1}, we have
		\begin{align}\label{7e5}
			H_{r,s,k}(z)\equiv \left(\sum_{n=0}^{\infty}\bar{b}^k_{r,s}(n)q^{n+p}\right)\cdot f_p^{24}\pmod{p}
		\end{align}
		and
		\begin{align}\label{7e7}
			I_{r,s,k}(z)\equiv \left(\sum_{n=0}^{\infty}\bar{b}^k_{r,s}(n)q^{n+p^2}\right)\cdot f_p^{24p}\pmod{p}.
		\end{align}
		From the proof of Theorem \ref{7t1}, it is easy to see that $H_{r,s,k}(z)\equiv F_{r,s,k}(z)\pmod{p}$. Thus, $H_{r,s,k}(z)\in M_{(24p^2-rk)/2}\left(\Gamma_0(N), \chi(\bullet)\right)$ with character $\chi(\bullet)=\left(\frac{(-1)^\ell 2^{sk}}{\bullet}\right)$. By Theorem \ref{t2.1}, we see that the level of $I_{r,s,k}(z)$ is the smallest positive integer $N$ such that 
		\begin{align*}
			N\left(\dfrac{(3s-2r)k}{2}+24p^2-2sk-\dfrac{(s-r)k}{4}\right)\equiv 0\pmod{24}.
		\end{align*}
		For $N$ as in \eqref{n1} and \eqref{n2}, the above identity is satisfied. It follows from Theorem \ref{t2.2} that $I_{r,s,k}(z)$ is holomorphic at a cusp $c/d$ if and only if
		\begin{align}\label{7e6}
			(3s-2r)k\cdot\dfrac{gcd(d,2)^2}{2}+24p^2-2sk-(s-r)k\cdot\dfrac{gcd(d,4)^2}{4}\ge 0.
		\end{align}
		For $d=1$ and $d=2$, \eqref{7e6} implies $	32p^2\ge k(r+s)$ and $8p^2\ge k(r-s)$, respectively, which are true since we have assumed $8p^2\ge k(r+s)$.
		For $d=4,8,16,32$, \eqref{7e6} implies $24p^2\ge 0$, which is always true. Thus, $I_{r,s,k}(z)$ is holomorphic at cusp $c/d$. The weight of $I_{r,s,k}(z)$ is $\ell=(24p^2-rk)/2$, which is a positive integer since $rk\equiv 0\pmod{2}$ and $rk<24p^2$. The associated character is given by $\chi(\bullet)=\left(\frac{(-1)^{\ell} 2^{sk}}{\bullet}\right)$.
		Hence, $I_{r,s,k}(z)\in M_{(24p^2-rk)/2}\left(\Gamma_0(N), \chi(\bullet)\right)$.\\
		Now, on applying $T_p$ operator to \eqref{7e5} and \eqref{7e7}, we obtain
		\begin{align*}
			H_{r,s,k}(z)\mid T_p\equiv \left(\sum_{n=0}^{\infty}\bar{b}^k_{r,s}(pn)q^{n+1}\right)\cdot f_1^{24}\pmod{p}
		\end{align*}
		and 
		\begin{align*}
			I_{r,s,k}(z)\mid T_p\equiv \left(\sum_{n=0}^{\infty}\bar{b}^k_{r,s}(pn)q^{n+p}\right)\cdot f_1^{24p}\pmod{p},
		\end{align*}
		respectively.\\
		Finally, on applying $T_p$ operator to $I_{r,s,k}(z)\mid T_p$, we obtain
		\begin{align*}
			I_{r,s,k}(z)\mid T^2_p\equiv \left(\sum_{n=0}^{\infty}\bar{b}^k_{r,s}(p^2n)q^{n+1}\right)\cdot f_1^{24}\pmod{p}. 
		\end{align*}
		The Sturm's bound for the associated space of modular forms is found to be  $ B:=\left(\dfrac{24p^2-rk}{24}\right)N \displaystyle\prod_{t\;prime, \;t\mid N}^{}\dfrac{t+1}{t}=\dfrac{(24p^2-rk)}{16}\cdot N$. Therefore, using Theorem \ref{sturm}, we complete the proof of Theorem \ref{7t3}.
	\end{proof}
	\begin{proof}[\textbf{Proof of Theorem \ref{7t1.1}-\ref{7t1.3}}]
		By Lemmas \ref{7t2}-\ref{7t3} it suffices to the check the congruences in Theorems \ref{7t1.1}-\ref{7t1.3} upto the Sturm's bound, which agrees on computational evidence using \textit{Mathematica}. This completes the proof.
	\end{proof}
	\section{Proofs of Theorems \ref{6tmodp1}-\ref{6tmod3.4}}\label{6s5}
	In this section, we establish congruences modulo prime for $\bar{b}_{r,s}^k(n)$ using elementary $q$-series techniques.
	\begin{proof}[\textbf{Proof of Theorem \ref{6tmodp1}}]
		From Lemma \ref{6l2}, we have
		\begin{align*}
			\sum_{n=0}^{\infty}\bar{b}^{k}_{p-s,s}(n)q^n&=\left(\phi(q)\right)^{sk}\prod_{i\ge1}^{}\left(\phi(q^{2^i})\right)^{p\cdot k \cdot2^i}\\
			&\equiv \phi^{t}(q^p)\phi(q)\prod_{i\ge1}^{}\left(\phi(q^{p\cdot2^i})\right)^{k\cdot2^i}\pmod{p},
		\end{align*}
		for some positive integer $t$, since $sk\equiv 1\pmod{p}$.
		Consider $pn+r\equiv \ell^2\pmod{p}$ which is true if and only if $r\equiv \ell^2\pmod{p}$. Since we have assumed $r$ is a quadratic non-residue modulo $p$, we get
		\begin{align}\label{6e1}
			\bar{b}^{k}_{p-s,s}(pn+r)\equiv 0\pmod{p}.
		\end{align} 
		Now, for all $m\ge j\ge 0$, from \eqref{6gf}, we have
		\begin{align*}
			\sum_{n=0}^{\infty}\bar{b}^{k}_{p(m-j)+p-s,pm+s}(n)q^n&=\dfrac{f_2^{(m+2j)pk}}{f_1^{2mpk}f_4^{jpk}}\cdot \dfrac{f_2^{(5s-2p)k}}{f_1^{2sk}f_4^{(2s-p)k}}\\
			&\equiv\dfrac{f_{2p}^{(m+2j)k}}{f_p^{2mk}f_{4p}^{jk}}\sum_{n=0}^{\infty}\bar{b}^k_{p-s,p}(n)q^n\pmod{p}.
		\end{align*}
		Extracting the terms of the form $q^{pn+r}$ and using \eqref{6e1}, we get the desired result.
	\end{proof}
	
	\begin{proof}[\textbf{Proof of Theorem \ref{6tmodp2}}]
		Let $p$ be a prime. Using \eqref{6gf} and Lemma \ref{bt}, we have
		\begin{align*}
			\sum_{n=0}^{\infty}\bar{b}^{p}_{r,s}(n)q^n=\dfrac{f_2^{(3s-2r)p}}{f_1^{2sp}f_4^{(s-r)p}}\equiv \dfrac{f_{2p}^{3s-2r}}{f_{p}^{2s}f_{4p}^{s-r}}\pmod{p}.
		\end{align*}
	Extracting the terms of the form $q^{pn+j}$ with $1\le j\leq p-1$ from the above congruence, yields the desired result.
	\end{proof}
	In the remainder of the present section all the congruences hold to the modulo 3.
	\begin{proof}[\textbf{Proof of Theorem \ref{6tmod3.1}}]
		Thanks to \eqref{6gf}, we have
		\begin{align*}
			\sum_{n=0}^{\infty}\bar{b}^k_{3r,3s}(n)q^n=\dfrac{f_2^{3k(3s-2r)}}{f_1^{3(2sk)}f_4^{3k(s-r)}}\equiv\dfrac{f_6^{(3s-2r)k}}{f_3^{2sk}f_{12}^{(s-r)k}}.
		\end{align*}
		Extracting the terms of the form $q^{3n+1}$ and $q^{3n+2}$ from the above congruence yields congruence \eqref{6e3n1.1}.\\
		Again, thanks to \eqref{6gf} and Lemma \ref{bt}, we have
		\begin{align}
			\sum_{n= 0}^{\infty}\bar{b}^{3k+2}_{3r+1,3s}(n)q^n&\equiv\dfrac{f_6^{(3k+2)(3s-2r)-(2k+1)}}{f_3^{2s(3k+2)}f_{12}^{(s-r)(3k+2)-k}}\cdot \dfrac{f_4^2}{f_2}\label{6e3.1}
		\end{align}
		and 
		\begin{align}
			\sum_{n=0}^{\infty}\bar{b}^{3k+1}_{3r+2,3s}(n)q^n&\equiv \dfrac{f_6^{(3s-r)(3k+1)-(4k+1)}}{f_3^{2s(3k+1)}f_{12}^{(s-r)(3k+1)-2k}}\cdot \dfrac{f_4^2}{f_2}.\label{6e3.2}
		\end{align}
		On using dissection \eqref{epsip} in \eqref{6e3.1} and \eqref{6e3.2}, and then extracting terms of the form $q^{3n+1}$ from the resulting congruences, we complete the proof of \eqref{6e3n1.2}.  \\
		From \eqref{6gf} and Lemma \ref{bt}, we have
		\begin{align}
			\sum_{n=0}^{\infty}\bar{b}^{3k+1}_{3r+2,3s+1}(n)q^n&\equiv \dfrac{f_6^{(3k+2)(3s+1-2r)-2(k+1)}}{f_3^{2(3k+2)+(2k+1)}f_{12}^{(s-r)(3k+2)}}\cdot \dfrac{f_2^2}{f_1}\label{6e3.3}
		\end{align}
		and
		\begin{align}
			\sum_{n=0}^{\infty}\bar{b}^{3k+1}_{3r+2,3s+2}(n)q^n&\equiv\dfrac{f_6^{(3s+2-2r)(3k+1)-(4k+2)}}{f_3^{2s(3k+1)+(4k+1)}f_{12}^{(s-r)(3k+1)}}\cdot \dfrac{f_2^2}{f_1}\label{6e3.4}.
		\end{align}
		Again, using \eqref{epsip} in \eqref{6e3.3} and \eqref{6e3.4}, and then extracting terms of the form $q^{3n+1}$ from the resulting congruences, we complete the proof of \eqref{6e3n2.1}.
	\end{proof}
	\begin{proof}[\textbf{Proof of Theorem \ref{6tmod3.2}}]
		Thanks to \eqref{6gf} and Lemma \ref{bt}, we have
		\begin{align*}
			\sum_{n=0}^{\infty}\bar{b}_{9r+3t+2, 9s+3u}^{9k+3}(n)q^n&\equiv\dfrac{f_6^{(3(9s+3u)-2(9r+3t+2))(3k+1)}}{f_3^{2(9s+3u)(3k+1)}f_{12}^{(9(s-r)+3(u-t)-2)(3k+1)}}\\
			\sum_{n=0}^{\infty}\bar{b}_{9r+3t+1,9s+3u}^{9k+6}(n)q^n&\equiv \dfrac{f_6^{(3(9s+3u)-2(9r+3t+1))(3k+2)}}{f_3^{2(9s+3u)(3k+2)}f_{12}^{(9(s-r)+3(u-t)-1)(3k+2)}}\\
			\sum_{n=0}^{\infty}\bar{b}_{9r+3t+2,9s+3u+2}^{9k+3}(n)q^n&\equiv \dfrac{f_6^{(3(9s+3u)-2(9r+3t)+2)(3k+1)}}{f_3^{2(9s+3u+2)(3k+1)}f_{12}^{(9(s-r)+3(u-t))(3k+1)}}\\
			\sum_{n=0}^{\infty}\bar{b}_{9r+3t+1,9s+3u+1}^{9k+6}(n)q^n&\equiv\dfrac{f_6^{(3(9s+3u)-2(9r+3t)+1)(3k+2)}}{f_3^{2(9s+3u+1)(3k+2)}f_{12}^{(9(s-r)+3(u-t))(3k+2)}}\\
			\sum_{n= 0}^{\infty}\bar{b}_{9r+3t+2,9s+3u+1}^{9k+3}(n)q^n&\equiv \dfrac{f_6^{(3(9s+3u)-2(9r+3t)-1)(3k+1)}}{f_3^{2(9s+3u+1)(3k+1)}f_{12}^{(9(s-r)+3(u-t)-1)(3k+1)}}\\
			\sum_{n=0}^{\infty}\bar{b}_{9r+3t+1,9s+3u+2}^{9k+6}(n)q^n&\equiv \dfrac{f_6^{(3(9s+3u)-2(9r+3t)+4)(3k+2)}}{f_3^{2(9s+3u+2)(3k+2)}f_{12}^{(9(s-r)+3(u-t)+1)(3k+2)}}.
		\end{align*}
		Extracting the terms of the form $q^{3n}$ from the above congruences, we obtain
		\begin{align}
			\sum_{n=0}^{\infty}\bar{b}_{9r+3t+2, 9s+3u}^{9k+3}(3n)q^n&\equiv\dfrac{f_6^{(3(3s+u)-2(3r+t)-1)(3k+1)-k}}{f_3^{2(3s+u)(3k+1)}f_{12}^{(3(s-r)+u-t)(3k+1)-2k}}\cdot\psi(q^2)\label{6e3.7}\\
			\sum_{n=0}^{\infty}\bar{b}_{9r+3t+1,9s+3u}^{9k+6}(3n)q^n&\equiv \dfrac{f_6^{(3(3s+u)-2(3r+t)-1)(3k+2)-k}}{f_3^{2(3s+u)(3k+2)}f_{12}^{(3(s-r)+u-t)(3k+2)-2k}}\cdot\psi(q^2)\label{6e3.8}\\
			\sum_{n=0}^{\infty}\bar{b}_{9r+3t+2,9s+3u+2}^{9k+3}(3n)q^n&\equiv \dfrac{f_6^{(3(3s+u)-2(3r+t))(3k+1)+2k}}{f_3^{2(3s+u)(3k+1)+4k+1}f_{12}^{(3(s-r)+u-t)(3k+1)}}\cdot \psi(q)\label{6e3.11}\\
			\sum_{n=0}^{\infty}\bar{b}_{9r+3t+1,9s+3u+1}^{9k+6}(3n)q^n&\equiv \dfrac{f_6^{(3(3s+u)-2(3r+t))(3k+2)+2k}}{f_3^{2(3s+u)(3k+2)+2k+1}f_{12}^{(3(s-r)+u-t)(3k+2)}}\cdot\psi(q)\label{6e3.12}\\
			\sum_{n= 0}^{\infty}\bar{b}_{9r+3t+2,9s+3u+1}^{9k+3}(3n)q^n&\equiv\dfrac{f_6^{(3(3s+u)-2(3r+t))(3k+1)-k}}{f_3^{2(3s+u)(3k+1)+2k+1}f_{12}^{(3(s-r)+u-t)(3k+1)-k}}\cdot \psi(-q)\label{6e3.15}\\
			\sum_{n=0}^{\infty}\bar{b}_{9r+3t+1,9s+3u+2}^{9k+6}(3n)q^n&\equiv\dfrac{f_6^{(3(3s+u)-2(3r+t)+1)(3k+2)+k+1}}{f_3^{2(3s+u)(3k+2)+4k+3}f_{12}^{(3(s-r)+u-t)(3k+2)+k+1}}\cdot\psi(-q)	\label{6e3.16}.
		\end{align}
		Employing \eqref{epsip} in congruences \eqref{6e3.7}-\eqref{6e3.8} and then extracting the terms of the form $q^{3n+1}$ from the resulting equations, we complete the proof of congruence \eqref{6e9.1}. Similarly, employing \eqref{epsip} in congruences \eqref{6e3.11}-\eqref{6e3.16} and then extracting the terms of the form $q^{3n+2}$ from the resulting equations, we complete the proof of congruences \eqref{6e9.2} and \eqref{6e9.3}.
	\end{proof}
	\begin{proof}[\textbf{Proof of Theorem \ref{6tmod3.3}}]
		As the proof of each elements of the sets
		$A_{t,u,\ell},B_{t,u,\ell}, C_{t,u,\ell},\\ D_{t,u,\ell},$ and  $E_{t,u,\ell}$ are similar, we provide the proof for only one representative from each set. Thanks to \eqref{6gf}, \eqref{epsi}, \eqref{epsim} and Lemma \ref{bt}, we have
		\begin{align*}
			\sum_{n=0}^{\infty}\bar{b}_{9r+1,9s+2}^{9k+5}(n)q^n&\equiv\dfrac{f_6^{3k(2(9s+2)-2(9s+1))+5((9s+2)-6r)-3}}{f_3^{6k(9s+2)+10(3s)+7}f_{12}^{3k(9s-9r+1)+5(3s-3r)+2}}\cdot \psi(-q)\\
			\sum_{n=0}^{\infty}\bar{b}_{9r+4,9s+1}^{9k+8}(n)q^n&\equiv\dfrac{f_6^{3k(3(9s+1)-2(9r+4))+8((9s+1)-6r)-11}}{f_3^{6k(9s+1)+16\cdot 3s+5}f_{12}^{3k(9s-9r-3)+8(3s-3r-1)}}\cdot \psi(q)\\
			\sum_{n=0}^{\infty}\bar{b}_{9r+2,9s}^{9k+7}(n)q^n&\equiv \dfrac{f_6^{3k(3(9s)-2(9r+2))+7(9s-6r)-9}}{f_3^{6s(9k+7)}f_{12}^{3k(9s-9r-2)+7(3s-3r)-4}}\cdot \psi(q^2)\\
			\sum_{n=0}^{\infty}\bar{b}_{9r+4,9s+2}^{9k+1}(n)q^n&\equiv\dfrac{f_6^{3k(3(9s)-2(9r)-2)+(9s-6r)-1}}{f_3^{6k(9s+2)+6s+1}f_{12}^{3k(9s-9r-2)+(3s-3r)-1}}\cdot \dfrac{1}{\psi(-q)}\\
			\sum_{n=0}^{\infty}\bar{b}_{9r+1,9s+4}^{9k+4}(n)q^n&\equiv\dfrac{f_6^{3k(3(9s)-2(9r)+10)+4(9s-6r)+13}}{f_3^{6k(9s+4)+8(3s)+11}f_{12}^{(3s-3r+1)(9k+4)}}\cdot F(q).
		\end{align*}
		Employing the dissections for $\psi(-q), \psi(q), \psi(q^2), 1/\psi(-q)$ and $F(q)$ from Lemma \ref{6l3d} in the above congruences accordingly and then extracting the terms of the form $q^{3n}$, we obtain
		\begin{align*}
			\sum_{n=0}^{\infty}\bar{b}_{9r+1,9s+2}^{9k+5}(3n)q^n&\equiv\dfrac{f_6^{k(3(9s+2)-2(9s+1))+5(3s-2r)+7}}{f_3^{2k(9s+2)+10s+4}f_{12}^{k(9s-9r+1)+5(s-r)+3}}\cdot \psi(q^2)\\
			\sum_{n=0}^{\infty}\bar{b}_{9r+4,9s+1}^{9k+8}(3n)q^n&\equiv\dfrac{f_6^{k(3(9s+1)-2(9s+4))+8(3s-2r)-2}}{f_3^{k(9s+1)+16s}f_{12}^{k(9s-9r-3)+8(s-r)-2}}\cdot \psi(q^2)\\
			\sum_{n=0}^{\infty}\bar{b}_{9r+2,9s}^{9k+7}(3n)q^n&\equiv\dfrac{f_6^{k(3(9s)-2(9r+2))+7(3s-2r)-1}}{f_3^{2s(9k+7)}f_{12}^{k(9s-9r-2)+7(s-r)}}\cdot \psi(q^2)\\
			\sum_{n=0}^{\infty}\bar{b}_{9r+4,9s+2}^{9k+1}(3n)q^n&\equiv\dfrac{f_6^{k(3(9s)-2(9r)-2)+(3s-2r)+8}}{f_3^{3k(9s+2)+2s+4}f_{12}^{k(9s-9r-2)+(s-r)+12}}\cdot \psi(q^2)\\
			\sum_{n=0}^{\infty}\bar{b}_{9r+1,9s+4}^{9k+4}(3n)q^n&\equiv\dfrac{f_6^{k(3(9s)-2(9r)+10)+4(3s-2r)+3}}{f_3^{2k(9s+u)+8s}f_{12}^{(s-r)(9k+4)+3k+2}}\cdot \psi(q^2).
		\end{align*}
		Further, again employing \eqref{epsip} in the above congruences and then extracting the terms of the form $q^{3n+1}$, we complete the proof.
	\end{proof}
	\begin{proof}[\textbf{Proof of Theorem \ref{6tmod3.4}}]
		As the proof of each elements of the sets
		$G_{t,u,\ell}, H_{t,u,\ell}, I_{t,u,\ell},\\ J_{t,u,\ell}, K_{t,u,\ell}, L_{t,u,\ell}, M_{t,u,\ell}, N_{t,u,\ell}, O_{t,u,\ell}, \text{ and } P_{t,u,\ell}$ are similar, we provide the proof for only one representative from each set. Thanks to \eqref{6gf}, \eqref{epsi}, \eqref{epsim} and Lemma \ref{bt}, we have
		\begin{align*}
			\sum_{n=0}^{\infty}\bar{b}_{9r+1,9s+5}^{9k+5}(n)q^n&\equiv\dfrac{f_6^{3k(3(9s+5)-2(9r+1))+5(9s+5-6r)-3}}{f_3^{6k(9s+5)+10(3s)+17}f_{12}^{3k(9s-9r+4)+5(3s-3r)+7}}\cdot \psi(-q)\\
			\sum_{n=0}^{\infty}\bar{b}_{9r+5,9s+4}^{9k+1}(n)q^n&\equiv\dfrac{f_6^{3k(3(9s+4)-2(9r+5))+(9s+4-6r)-3}}{f_3^{6k(9s+4)+6s+3}f_{12}^{3k(9s-9r-1)+3(s-r)}}\cdot \psi(-q)\\
			\sum_{n=0}^{\infty}\bar{b}_{9r+1,9s+1}^{9k+5}(n)q^n&\equiv \dfrac{f_6^{3k(3(9s+1)-2(9r+1))+5(9s+1-6r)-4}}{f_3^{6k(9s+1)+10(3s)+3}f_{12}^{(3s-3r)(9k+5)}}\cdot \psi(q)\\
			\sum_{n=0}^{\infty}\bar{b}_{9r+5,9s+2}^{9k+1}(n)q^n&\equiv\dfrac{f_6^{3k(3(9s+2)-2(9r+5))+(9s+2-6r)-4}}{f_3^{3k(2(9s+2))+6s+1}f_{12}^{(3s-3r-1)(9k+1)}}\cdot \psi(q)\\
			\sum_{n=0}^{\infty}\bar{b}_{9r+1,9s+3}^{9k+5}(n)q^n&\equiv\dfrac{f_6^{3k(3(9s+3)-2(9r+1))+5(9s+3-6r)-3}}{f_3^{6k(9s+3)+10(3s+1)}f_{12}^{3k(9s-9r+2)+5(3s-3r)+4}}\cdot \psi(q^2)\\
			\sum_{n=0}^{\infty}\bar{b}_{9r+5,9s+3}^{9k+1}(n)q^n&\equiv \dfrac{f_6^{3k(2(9s+3)-2(9r+5))+(9s+3-6r)-3}}{f_3^{6k(9s+3)+2(3s+1)}f_{12}^{3k(9s-9r-2)+(3s-3r)}}\cdot \psi(q^2)\\
			\sum_{n=0}^{\infty}\bar{b}_{9r+4,9s+5}^{9k+1}(n)q^n&\equiv\dfrac{f_6^{3k(3(9s)-2(9r)+7)+(9s-6r)+2}}{f_3^{6k(9s+5)+2(3s)+3}f_{12}^{3k(9s-9r+1)+(3s-3r)}}\cdot \dfrac{1}{\psi(-q)}\\
			\sum_{n=0}^{\infty}\bar{b}_{9r+1,9s+2}^{9k+4}(n)q^n&\equiv\dfrac{f_6^{3k(3(9s)-2(9r)+4)+4(9s-6r)+5}}{f_3^{6k(9s+2)+8(3s)+5}f_{12}^{3k(9s-9r+1)+4(3s-3r)+1}}\cdot\dfrac{1}{\psi(-q)}\\
			\sum_{n=0}^{\infty}\bar{b}_{9r+4,9s+1}^{9k+1}(n)q^n&\equiv\dfrac{f_6^{3k(3(9s)-2(9r)-5)+(9s-6r)-2}}{f_3^{6k(9s+1)+2(3s)+1}f_{12}^{(3s-3r-1)(9k+1)}}\cdot F(q)\\
			\sum_{n=0}^{\infty}\bar{b}_{9r+1,9s+1}^{9k+4}(n)q^n&\equiv\dfrac{f_6^{3k(3(9s)-2(9r)+1)+4(9s-6r)+1}}{f_3^{6k(9s+1)+8(3s)+3}f_{12}^{(3s-3r)(9k+4)}}\cdot F(q).
		\end{align*}
		Employing the dissections for $\psi(-q), \psi(q), \psi(q^2), 1/\psi(-q)$ and $F(q)$ from Lemma \ref{6l3d} in the above congruences accordingly and then extracting the terms of the form $q^{3n}$, we obtain
		\begin{align*}
			\sum_{n=0}^{\infty}\bar{b}_{9r+1,9s+5}^{9k+5}(3n)q^n&\equiv\dfrac{f_6^{k(3(9s+5)-2(9r+1))+5(3s-2r)+11}}{f_3^{k(9s+5)+10s+7}f_{12}^{k(9s-9r+4)+5(s-r)+4}}\cdot \psi(q)\\
			\sum_{n= 0}^{\infty}\bar{b}_{9r+5,9s+4}^{9k+1}(3n)q^n&\equiv\dfrac{f_6^{k(3(9s+4)-2(9r+5))+(3s-2r)+5}}{f_3^{2k(9s+4)+2s+3}f_{12}^{k(9s-9r-1)+(s-r)+2}}\cdot \psi(-q)\\
			\sum_{n=0}^{\infty}\bar{b}_{9r+1,9s+1}^{9k+5}(3n)q^n&\equiv\dfrac{f_6^{k(3(9s+1)-2(9r+1))+5(3s-2r)-1}}{f_3^{2k(9s+1)+10s-1}f_{12}^{(9k+5)(s-r)}}\cdot\psi(q)\\
			\sum_{n=0}^{\infty}\bar{b}_{9r+5,9s+2}^{9k+1}(3n)q^n&\equiv\dfrac{f_6^{k(3(9s+2)-2(9r+5))+(3s-2r)-1}}{f_3^{k(2(9s+2))+3s-1}f_{12}^{9k(s-r)+(s-r)}}\cdot \psi(-q)\\
			\sum_{n=0}^{\infty}\bar{b}_{9r+1,9s+3}^{9k+5}(3n)q^n&\equiv\dfrac{f_6^{k(3(9s+3)-2(9r+1))+5(3s+1-2r)}}{f_3^{2k(9s+3)+10s+3}f_{12}^{k(9s-9r+2)+5(s-r)+2}}\cdot \psi(q)\\
			\sum_{n=0}^{\infty}\bar{b}_{9r+5,9s+3}^{9k+1}(3n)q^n&\equiv\dfrac{f_6^{k(3(9s+3)-2(9r+5))+(3s+1-2r)+1}}{f_3^{2k(9s+3)+2s+1}f_{12}^{k(9s-9r-2)+(s-r)+1}}\cdot \psi(-q)
				\end{align*}
			\begin{align*}
			\sum_{n=0}^{\infty}\bar{b}_{9r+4,9s+5}^{9k+1}(3n)q^n&\equiv \dfrac{f_6^{k(3(9s)-2(9r)+7)+(3s-2r)+9}}{f_3^{k(9s+5)+2s+5}f_{12}^{k(9s-9r+1)+(s-r)+12}}\cdot \psi(-q)\\
			\sum_{n=0}^{\infty}\bar{b}_{9r+1,9s+2}^{9k+4}(3n)q^n&\equiv\dfrac{f_6^{k(3(9s)-2(9r)+4)+4(9s-6r)+9}}{f_3^{2k(9s+2)+8s+5}f_{12}^{k(9s-9r+1)+4(s-r)+12}}\cdot \psi(q)\\
			\sum_{n=0}^{\infty}\bar{b}_{9r+4,9s+1}^{9k+1}(3n)q^n&\equiv\dfrac{f_6^{k(3(9s)-2(9r)-5)+(3s-2r)-2}}{f_3^{2k(9s+1)+2s-3}f_{12}^{(s-r)(9k+1)-3k}}\cdot \psi(-q)\\
			\sum_{n=0}^{\infty}\bar{b}_{9r+1,9s+1}^{9k+4}(3n)q^n&\equiv \dfrac{f_6^{k(3(9s)-2(9r)+1)+4(3s-2r)-2}}{f_3^{k(9s+1)+8s-3}f_{12}^{(s-r)(9k+4)}}\cdot \psi(q).
		\end{align*}
		On using \eqref{epsip} in the above congruences and then finally extracting the terms of the form $q^{3n+2}$, we complete the proof.
	\end{proof}
	\section{Concluding remarks}\label{6s7}
	We conclude the paper with the following remarks.
	\begin{enumerate}
		\item In this work, we proved several congruences using the theory of modular forms. It would be nice to see elementary proofs of these congruences.
			\item Similar to congruences modulo $3$ for progressions  $3n+1, 3n+2, 9n+3$ and $9n+6$ as in Theorems \ref{6tmodp1}-\ref{6tmod3.4}, the results of Chacon and Sellers \cite[Theorems~4.5-4.8]{chacon} can be generalized as stated below and proved using identical arguments. 
		\begin{theorem}
			For all $r,s,k\ge 0$ and $n\ge 0$, 
			\begin{align*}
				\bar{b}_{27r+2,27s+1}^{27k+5}(27n+9)&\equiv\bar{b}_{27r+2,27s+1}^{27k+5}(27n+18)\equiv 0\pmod{3}\\
				\bar{b}_{27r+2,27s+1}^{27k+8}(27n+18)&\equiv 0\pmod{3}\\
				\bar{b}_{27r+2,27s+1}^{27k+10}(27n+18)&\equiv 0\pmod{3}\\
				\bar{b}_{27r+2,27s+1}^{27k+14}(27n+18)&\equiv 0\pmod{3}.
			\end{align*}
		\end{theorem}
		\noindent It is natural to expect an extension of these results to other families. However, deriving families similar to Theorems \ref{6tmod3.2}-\ref{6tmod3.4} for progressions $27n+9$ and $27n+18$ requires considerably more elaborate computations. We therefore leave investigation of such congruence families as an open problem.
		\item Our primary motivation was to extend the list of cogruences proved by Chacon and Sellers \cite{chacon} to $\bar{b}_{r,s}^k(n)$. However, based on computational evidences we note that there exists other congruences for odd moduli apart from the ones proved in this paper. We leave it to the motivated reader to prove the congruences in the following conjecture.
		\begin{conjecture}
			For all $n\ge 0$ and \\
			1. for $(r,s,k)\in\{(5,2,2),(4,4,5), (1,2,2), (3,1,6), (4,1,5) \}$,
			\begin{align*}
				\bar{b}_{r,s}^k(25n+10)\equiv 0\pmod{5}.
			\end{align*}
			2. for $(r,s,k)\in\{(5,3,2), (4,5,5), (1,3,2) , (4,1,9)\}$,
			\begin{align*}
				\bar{b}_{r,s}^k(25n+20)\equiv 0\pmod{5}.
			\end{align*}
		\end{conjecture}
		\end{enumerate}
	\noindent\textbf{Acknowledgment.} The authors would like to thank Professor Manjil P. Saikia for many valuable comments on the paper.
		
	\bigskip
	\bigskip
	
	\noindent
	Department of Mathematics\\
	Ramanujan School of Mathematical Sciences\\
	Pondicherry University\\
	Puducherry- 605 014, India.\\

	\noindent Email: \texttt{tthejithamp@pondiuni.ac.in}
	
	\noindent	Email: \texttt{dr.fathima.sn@pondiuni.ac.in} (\Letter)
\end{document}